\documentclass[oneside,english,12pt,reqno]{amsart}
\usepackage[colorlinks=true,linkcolor=cyan,citecolor=magenta]{hyperref}
\usepackage{comment}

\allowdisplaybreaks

\usepackage{units}
\usepackage{mathrsfs}
\usepackage{amstext}
\usepackage{amsthm}
\usepackage{esint}
\usepackage{amsmath}
\usepackage{amssymb}
\usepackage{soul}
\usepackage{cancel}
\newcommand{\R}{\mathbb{R}}
\newcommand{\pa}{\partial}
\newcommand{\eps}{\varepsilon}
\usepackage{a4wide}
\usepackage{enumitem}
\newcommand{\bra}[1]{\left(#1\right)}
\newcommand{\sbra}[1]{\left[#1\right]}

\newcommand{\pt}{\varphi_\tau}
\newcommand{\na}{\nabla}

\newcommand{\lam}{\lambda}

\newcommand{\F}{\mathscr{F}}

\newcommand{\red}{\textcolor{red}}
\newcommand{\blue}{\textcolor{black}}
\newcommand{\wh}{\widehat}

\newcommand{\intO}{\int_{\Omega}}

\newcommand{\intQT}{\int_0^T\int_{\Omega}}

\renewcommand{\L}{\mathscr{L}}
\renewcommand{\H}{\mathscr{H}}
\newcommand{\E}{\mathscr{E}}

\newcommand{\vt}{\vartheta}
\newcommand{\ue}{u^{\eps}}

\newcommand{\sumi}{\sum_{i=1}^{m}}
\newcommand{\sumj}{\sum_{j=1}^{m}}
\newcommand{\sumk}{\sum_{k=1}^{m}}

\newcommand{\LO}[1]{L^{#1}(\Omega)}

\newcommand{\LQ}[1]{L^{#1}(Q_T)}

\makeatletter
\numberwithin{equation}{section}
\numberwithin{figure}{section}

\newtheorem{theorem}{Theorem}[section]
\newtheorem{lemma}{Lemma}[section]
\newtheorem{remark}{Remark}[section]
\newtheorem{definition}{Definition}[section]
  
\makeatother

\usepackage{babel}

\newtheorem{proposition}{Proposition}[section]

\author[W.E. Fitzgibbon]{William E Fitzgibbon}
\address{William E Fitzgibbon \hfill\break
	Department of Mathematics, 
	University of Houston, Houston, Texas 77204, USA}
\email{fitz@uh.edu}

\author[J. Morgan]{Jeff Morgan}
\address{Jeff Morgan \hfill\break
	Department of Mathematics, 
	University of Houston, Houston, Texas 77204, USA}
\email{jmorgan@math.uh.edu}

\author[B.Q. Tang]{Bao Quoc Tang}
\address{Bao Quoc Tang \hfill\break
	Institute of Mathematics and Scientific Computing, University of Graz, 
	Heinrichstrasse 36, 8010 Graz, Austria}
\email{quoc.tang@uni-graz.at, baotangquoc@gmail.com}

\author[H.-M. Yin]{Hong-Ming Yin}
\address{Hong Ming Yin\hfill\break
	Department of Mathematics, Washington State University, Pullman, WA, 99164, USA}
\email{hyin@wsu.edu}

\title[Reaction-Diffusion-Advection Systems with Non-smooth Diffusion]{Reaction-Diffusion-Advection Systems with Discontinuous Diffusion and Mass Control}

\begin{document}
\begin{abstract}
	In this paper, we study unique, globally defined uniformly bounded weak solutions for a class of semilinear reaction-diffusion-advection systems. The coefficients of the differential operators and the initial data are only required to be measurable and uniformly bounded. The nonlinearities are quasi-positive and satisfy a commonly called mass control or dissipation of mass property. Moreover, we assume the intermediate sum condition of a certain order. The key feature of this work is the surprising discovery that quasi-positive systems that satisfy an intermediate sum condition automatically give rise to a new class of $L^p$-energy type functionals that allow us to obtain requisite uniform a priori bounds. Our methods are sufficiently robust to extend to different boundary conditions, or to certain quasi-linear systems. We also show that in case of mass dissipation, the solution is bounded in sup-norm uniformly in time. We illustrate the applicability of results by showing global existence and large time behavior of models arising from a spatio-temporal spread of infectious disease.
\end{abstract}

\maketitle
\tableofcontents

\section{Introduction}
\subsection{Problem setting}
Let $1\leq n\in \mathbb N$, and $\Omega\subset\R^n$ be a bounded domain with Lipschitz boundary $\pa\Omega$ such that $\Omega$ lies locally on one side of $\pa\Omega$. \textit{This condition is assumed throughout the current paper}.
Let $2\leq m\in \mathbb N$. We consider the following reaction-diffusion system for vector of concentrations $u = (u_1, u_2, \ldots, u_m)$: for any $i\in \{1,\ldots, m\}$,
\begin{equation}\label{Sys}
\begin{cases}
	\pa_t u_i - \na\cdot(D_i(x,t)\na u_i) + \na\cdot(B_i(x,t) u_i) = F_i(x,t,u), &x\in\Omega, t>0,\\
	u_i(x,t) = 0, &x\in\pa\Omega, t>0,\\
	u_i(x,0) = u_{i,0}(x), &x\in\Omega,
\end{cases}
\end{equation}
where the initial data $u_{i,0}$ is bounded and non-negative, the diffusion matrix $D_i: \Omega\times [0,\infty) \to \R^{n\times n}$ satisfies
\begin{equation}\label{ellipticity}
	\lambda |\xi|^2 \leq \xi^{\top}D_i(x,t)\xi \quad \forall (x,t)\in \Omega\times[0,\infty), \; \forall \xi \in \R^n, \; \forall i=1,\ldots, m,
\end{equation}
for some $\lam>0$, and for each $T>0$,
\begin{equation}\label{bounded_diffusion}
	D_i \in L^{\infty}(\Omega\times(0,T))  \quad\forall i=1,\ldots, m,
\end{equation}
and the drift $B_i: \Omega\times[0,\infty) \to \R^n$ is bounded, i.e.
\begin{equation}\label{bounded_drift}
	{\rm{ess}\sup}_{x,t}|B_i| \leq \Gamma \quad \forall i=1,\ldots, m,
\end{equation}
for some $\Gamma>0$. The nonlinearities $F_i:\Omega\times\R_+\times\R_+^m\to \R$ satisfy the following conditions:
\begin{enumerate}[label=(F\theenumi),ref=F\theenumi]
	\item\label{F1} for any $i=1,\ldots, m$ and any $(x,t)\in\Omega\times\R_+$, $F_i(x,t,\cdot):\R^m \to \R$ is locally Lipschitz continuous uniformly in $(x,t)\in \Omega\times(0,T)$ for any $T>0$;
	\item\label{F2} for any $i=1,\ldots, m$, any $(x,t)\in \Omega\times\R_+$, $F_i(x,t,\cdot)$ is quasi-positive, i.e. $F_i(x,t,u)\ge 0$ for all $u\in\R_+^m$ with $u_i =0$, for all $i=1,\ldots, m$;
	\item\label{F3} there exist $c_1,\ldots, c_m>0$ and $K_1, K_2\in \R$ such that
	\begin{equation*}
		\sumi c_iF_i(x,t,u) \leq K_1\sumi u_i + K_2 \quad \forall (x,t,u)\in \Omega\times\R_+\times \R_+^m;
	\end{equation*}
	\item\label{F4} there exist $K_3>0$, $r>0$, and a lower triangular matrix $A = (a_{ij})$ with positive diagonal entries, and non-negative entries otherwise, such that, for any $i=1,\ldots, m$,
	\begin{equation*}
		\sum_{j=1}^i a_{ij}F_j(x,t,u) \leq K_3\bra{1+\sumi u_i^{r}} \quad \forall (x,t,u)\in \Omega\times\R_+\times\R_+^m
	\end{equation*}
	(we call this assumption {\it intermediate sum of order $r$});
	\item\label{F5} the nonlinearities are bounded above by a polynomial, i.e. there exist $\ell>0$ and $K_4>0$ such that
	\begin{equation*}
		F_i(x,t,u) \leq K_4\bra{1+\sumi u_i^{\ell}}, \quad \forall (x,t,u)\in \Omega\times\R_+\times\R_+^m, \; \forall i=1,\ldots, m.
	\end{equation*}
\end{enumerate}
	The quasi-positivity of nonlinearities in assumption \eqref{F2} has a simple physical interpretation. If the concentration $u_i = 0$, it cannot be consumed in the reaction. Mathematically, this assumption implies that if the initial data is non-negative, then so is the solution, as long as it exists. At first glance, we note that \eqref{F5} implies \eqref{F4} with $r = \ell$. However, our results only place restrictions on the size of $r$ in \eqref{F4}, \textit{while $\ell$ in \eqref{F5} can be arbitrarily large}. In addition, we note that \eqref{F4} does not imply that the vector field grows at most as a power of $r$. It simply implies that $F_1$ is bounded above by a polynomial of degree $r$, and that higher order terms in subsequent $F_j$ have a ``canceling effect". \blue{This assumption holds naturally in many systems arising from chemistry or biology, see e.g. \cite{morgan2020boundedness,morgan2021global}. Roughly speaking, for any reaction/interaction, a gain term for a species stems from a loss term from another species, which results in the same terms in the nonlinearities for these species \textit{but with opposite signs}. This leads to the desired ``canceling effect" described in assumption \eqref{F4}.}

	\medskip
	Our focus is the global existence and uniform boundedness of solutions to \eqref{Sys} and its variance under the previous assumptions\footnote{Note that except for \eqref{F1}, \eqref{F2} and \eqref{F5}, some of our results do not assume \eqref{F3} and \eqref{F4}. Precise assumptions will be explicitly stated in each lemma or theorem.}. The main results, together with their proof methods and our key new ideas are detailed in the following section.

\subsection{State of the art and Motivation}
Reaction-diffusion systems with \eqref{F1}, quasi-positivity condition \eqref{F2} and control of mass \eqref{F3} appear frequently in physical, chemical or biological models, and the study of global well-posedness for such systems has produced an extensive literature in the last four decades, see e.g. \cite{rothe1984global,hollis1987global,fitzgibbon1997stability,morgan1990boundedness,pierre2010global} and references therein. In the spatially homogeneous case, i.e. the diffusion and the advection in \eqref{Sys} are neglected, these three conditions immediately imply the global existence, uniqueness and uniform-in-time bound (in case $K_1=K_2=0$ or $K_1<0$ in \eqref{F3}) of solutions. In case of spatially inhomogeneous systems where the diffusion is present, the situation is much more challenging. In fact, it was shown in \cite{pierre2000blowup} that \eqref{F2} and \eqref{F3} are not enough to prevent solutions to \eqref{Sys} from blowing up (in sup-norm) in finite time. A lot of effort has subsequently been spent on systems satisfying \eqref{F2} and \eqref{F3}, and nonlinearities having polynomial growth, i.e.
\begin{equation*}
	|f_i(u)| \lesssim 1 + |u|^{r} \quad \forall i=1,\ldots, m,
\end{equation*}
for some $r\geq 1$. In \cite{goudon2010regularity}, it was shown under an additional assumption called {\it entropy inequality} that if $r = 3$ and $n=1$, or $r=2$ and $n=2$, then the local bounded solution exists globally. This was later extended for strictly sub-quadratic growth $r<2$ in all dimensions in \cite{morgan1989global,caputo2009global}. Without assuming the entropy inequality, \cite{morgan2004global,canizo2014improved} showed that systems with quadratic nonlinearities, i.e. $r=2$, possess global bounded solutions when $n\leq 2$. This was latter shown also in \cite{pierre2017dissipative}, and improved in \cite{morgan2020boundedness} where the growth condition was replaced by a weaker intermediate sum condition \eqref{F4}. The global existence of quadratic systems in higher dimensions had been open until recently when it was settled in three parallel works \cite{souplet2018global,caputo2019solutions,fellner2020global}, in the first two works the entropy condition was still imposed, and it was removed completely in the last work. There is also work concerning weak solutions. The reader is referred to \cite{pierre2010global} for an extensive survey.

\medskip
We emphasize that most, if not all, of the existing literature consider the case of constant or smooth diffusion coefficients. There is a technical reason. For constant or smooth diffusion coefficients, one can utilize the duality method (see e.g. \cite{pierre2010global,morgan2020boundedness}) to first obtain initial a-priori estimates in $L^p(\Omega\times(0,T))$ and to extend these estimates from one component to another, and then the regularizing effect of the heat operator helps to initiate a bootstrap argument which ultimately leads to boundedness in $L^{\infty}(\Omega\times(0,T))$, and hence global existence. The case of inhomogeneous diffusion coefficients has been studied much less frequently, see e.g. \cite{desvillettes2007global} or \cite{bothe2017global}. Moreover, this work requires smoothness such as continuity or even differentiability of the diffusion coefficients. This motivates the study of the current paper where we investigate the global well-posedness of \eqref{Sys} without assuming any regularity of diffusion coefficients $D_i(x,t)$ other than their ellipticity and boundedness. We also remark that system \eqref{Sys} includes advection, which has been frequently neglected in the literature.

\medskip
In many cases, advection, diffusion or reaction processes take place in highly heterogeneous domains. If the terms in the differential operators and/or the reaction terms have discontinuity, then we cannot expect well-posedness in the classical sense,  but rather weak solutions in $L^p$ spaces. Concerning spatial inhomogeneity, the classical book \cite{LSU68} considered diffractive differential operators, which are piecewise smooth on sub regions of $\Omega$ and satisfy certain compatibility conditions that guarantee continuity of the state variables and their flux but not their gradients.  Strong solutions to systems with diffractive differential operators modelling the spatio-temporal spread of diseases among animal populations distributed in heterogeneous environments are obtained in \cite{fitzgibbon2001mathematical,fitzgibbon2004reaction}. What distinguishes the work at hand from previous studies is that we do not place smoothness, piecewise smoothness, or even continuity conditions on the advection and diffusion coefficients. This high degree of heterogeneity precludes strong $L^p$ solutions and forces us to consider weak solutions instead. 
The lack of regularity renders the aforementioned duality method non-applicable. In this paper, \textit{we overcome this issue by introducing a new family of $L^p$-energy functions} which, in combination with the intermediate sum condition, provides suitable a-priori estimates to obtain global existence and uniform-in-time boundedness of solutions to \eqref{Sys}. Note that the uniform-in-time bound of solutions can be applied to determine the large time behavior of the corresponding system. We believe that our results are applicable in a variety of scenarios. To illustrate, we apply our results to an infectious disease model.
\subsection{Main results and key ideas}
We will introduce the notion of weak solutions which is defined in the following.
\begin{definition}\label{def_weak}
	A vector of non-negative state variables $u = (u_1,\ldots, u_m)$ is called a weak solution to \eqref{Sys} on $(0,T)$ if
	\begin{equation*}
		u_i \in C([0,T];\LO{2})\cap L^2(0,T;H_0^1(\Omega)), \quad F_i(u) \in L^2(0,T;\LO{2}),
	\end{equation*}
	with $u_i(\cdot,0) = u_{i,0}(\cdot)$ for all $i=1,\ldots, m$, and for any test function $\varphi\in L^2(0,T;H_0^1(\Omega))$ with $\pa_t\varphi \in L^2(0,T;H^{-1}(\Omega))$, one has
	\begin{align*}
		&\intO u_i(x,t)\varphi(x,t)dx\bigg|_{t=0}^{t=T} - \intQT u_i\pa_t\varphi dxdt + \intQT D_i(x,t)\na u_i\cdot \na \varphi dxdt\\
		&=  \intQT u_iB_i(x,t)\cdot \na\varphi dxdt + \intQT F_i(u)\varphi dxdt.
	\end{align*}
\end{definition}

Our first main result is the global existence and uniform boundedness of system \eqref{Sys} under, among others, the assumption on control of mass \eqref{F3}.
\begin{theorem}\label{thm1}
	Assume \eqref{ellipticity}, \eqref{bounded_diffusion}, \eqref{bounded_drift}, \eqref{F1}, \eqref{F2}, \eqref{F3}, \eqref{F4} and \eqref{F5}. Assume moreover that
	\begin{equation}\label{growth}
		0\leq r < 1 + \frac{2}{n}.
	\end{equation}
	Then for any non-negative, bounded initial data $u_0\in \LO{\infty}^m$, there exists a unique global weak solution to \eqref{Sys} with $u_i\in L^\infty_{\rm{loc}}(0,\infty;\LO{\infty})$ for all $i=1,\ldots, m$. Moreover, if $K_1< 0$ or $K_1= K_2 = 0$, then the solution is bounded uniformly in time, i.e. 
	\begin{equation}\label{Linftybound}
		{\rm{ess}\sup}_{t\geq 0}\|u_i(t)\|_{\LO{\infty}} < +\infty, \quad \forall i=1,\ldots, m.
	\end{equation}
\end{theorem}
\begin{remark}\hfill
\begin{itemize}
	\item \blue{When $K_1 = K_2 = 0$, the assumption \eqref{F3} becomes the well known {\normalfont mass dissipation}, which has been considered frequently in the literature. The case $K_1<0$  can occur in chemical reactions when there is some continuous source of reactants. An example is the Gray-Scott system.}
	\item The uniform-in-time bound \eqref{Linftybound} is shown by using the fact that the $L^1(\Omega)$-norm of the solution is bounded uniformly in time, which is inferred from the assumption $K_1<0$ or $K_1 = K_2 = 0$. If the $L^1(\Omega)$-norm can be shown to be bounded uniformly in time by some other way, e.g. using special structures of the system at hand, then the condition $K_1<0$ or $K_1 = K_2 = 0$ can be relaxed.
\end{itemize}
\end{remark}
The proof of Theorem \ref{thm1} is based on a new $L^p$-energy approach. Traditionally, one seeks an energy function of \eqref{Sys} of the form
\begin{equation*}
	\E[u] = \sumi \intO h_i(u_i)dx
\end{equation*}
which is decreasing or at least bounded in time. When $h_i(z) \sim z^p$, it yields an $L^p$-estimate of the solution, which, for $p$ large enough, would eventually lead to bounds in $L^\infty$-norm. Such an approach is, unfortunately, very likely to fail under the general assumptions \eqref{F3}-\eqref{F4}, except for some very special cases. There is another approach called the {\it duality method} (see \cite{pierre2010global,canizo2014improved,morgan2020boundedness}) which has proved very efficient when dealing with systems with {\it constant or smooth diffusion coefficients}. Using this method, one gets from the mass control condition \eqref{F3} an $L^{2+\eps}(\Omega\times(0,T))$-estimate. This initial estimate and another duality argument are sufficient to bootstrap the regularity of the solution, using the intermediate sum condition \eqref{F4} and the growth assumption \eqref{growth}, to eventually obtain bounds in $L^\infty(\Omega\times(0,T))$ which ensure global existence. We refer the reader to \cite{morgan2020boundedness} for more details. However, this method seems not extendable to the case of merely bounded measurable diffusion coefficients, unless some additional regularity assumptions are imposed (see \cite{desvillettes2007global,bothe2017global}).

\medskip
The main idea of the $L^p$-approach in this paper is to look for an energy function consisting of {\it mixed polynomials} of order  $p\in \mathbb N$, with well chosen coefficients, namely
\begin{equation*}
	\L_p[u] = \intO\sum_{\text{deg}(Q) = p} \theta_QQ[u] dx,
\end{equation*}
where $\theta_Q>0$ depends on the \textit{monomial} $Q[u]$. Thanks to the non-negativity of the solution and the choice of $\theta_Q$, $(\L_p[u])^{1/p}$ is an equivalent $L^p$-norm. Of course, the main challenge is to find an algorithm for choosing the coefficients $\theta_Q$ in a manner that it is {\it compatible} to both diffusion and reactions in the sense that the evolution of $\L_p[u]$ is well behaved. In the case at hand, we show that the intermediate sum condition \eqref{F4} allows us to choose monomials of the form
\begin{equation*}
	Q[u] = \prod_{i=1}^{m}u_i^{\beta_i} \text{ for } \beta_i\in \mathbb N_0 \text{ satisfying } \sum_{i=1}^{m}\beta_i = p,
\end{equation*}
and coefficients of the form
\begin{equation*}
	\theta_Q = \frac{p!}{\beta_1!\ldots \beta_m!}\prod_{i=1}^{m}\theta_i^{\beta_i^2}
\end{equation*}
with appropriately chosen $\theta_1, \ldots, \theta_m>0$. We note that preliminary ideas of this $L^p$-energy approach have been used previously, see \cite{malham1998global,kouachi2001existence}. A recent work \cite{morgan2021global} by the second and third authors has also used this method in the context of volume-surface systems with constant diffusion coefficients. \blue{The novelty of the present work is to significantly extend this $L^p$-approach to the case of nonsmooth diffusion (and advection) coefficients. This, in particular, requires non-trivial extensions in studying the properties of $\L_p[u]$ (see Lemmas \ref{Lp-bound} and \ref{Hp-lem8}). Moreover, we show that this method is sufficiently robust to model variants and different boundary conditions.} It remains as an interesting open issue if the choice of the coefficients $\theta_Q$ is purely mathematical or due to deeper structure of the system under consideration.

\medskip
It is worth noting that other conditions can also lead to a priori estimates. For example, the so-called entropy condition (see e.g. \cite{goudon2010regularity,souplet2018global,fischer2015global}) has been employed in a number of papers. In this case, one requires the existence of scalars $\mu_i\in\mathbb{R}$ so that 
\[
\sumi F_i(x,t,u)\left(\log(u_i)+\mu_i\right)\le 0\text{ for all }(x,t,u)\in\Omega\times\mathbb{R}_+\times\mathbb{R}_+^m.
\]
This condition guarantees an $L_1(\Omega)$ a priori estimate for $H(u(\cdot,t))$ where
\[
H(u)=\sum_{i=1}^m u_i\bra{\log u_i - 1 + \mu_i}\quad \text{ for all }\quad u_i\geq 0.
\]
In this case of the entropy condition, we could obtain a priori estimates by multiplying the PDE for $u_i$ in (\ref{Sys}) by $\log(u_i)+\mu_i$, integrating over $\Omega$ and summing the results. Note that this takes advantage of the positivity of the second derivative of $u_i(\log u_i -1 + \mu_i)$. More generally, we could assume the existence of a set $M=\prod_{k=1}^m(\alpha_i,\beta_i)$ where $\alpha_i,\beta_i$ are extended real numbers such that $\alpha_i<\beta_i$ for each $i=1,\ldots, m$, and solutions to (\ref{Sys}) remain in $M$ whenever initial data lies in $M$, a function $H:M\to \mathbb{R}_+$ that is $C^2$ and has the form
\[
H(u)=\sumi h_i(u_i)
\]
where $h_i:(\alpha_i,\beta_i)\to\mathbb{R}_+$ satisfies
\begin{equation}\label{H1}\tag{H1}
\begin{gathered}
h_i''(z) \ge 0 \quad \text{ for all } z\in (\alpha_i,\beta_i),\\
h_i(z) \text{ is bounded implies } z \text{ is bounded},\\
\nabla H(u)\cdot F(x,t,u)\le K_5\sumi h_i(u_i)+K_6\text{ for all }(x,t,u)\in\Omega\times\mathbb{R}_+\times\mathbb{R}_+^m,
\end{gathered}
\end{equation}
for some $K_5, K_6>0$.
Then, analogous to above, we would obtain an $L_1(\Omega)$ a priori estimate for $H(u(\cdot,t))$. In addition, intermediate sum conditions could also be written in the form 
\begin{equation}\label{H2}\tag{H2}
A\begin{pmatrix}h_1'(u_1)F_1(x,t,u)\\\vdots\\h_m'(u_1)F_m(x,t,u)\end{pmatrix}\le K_7\vec{1}\left(\sumi h_i(u_i)+1\right)^r\text{ for all }(x,t,u)\in\Omega\times\mathbb{R}_+\times\mathbb{R}_+^m,
\end{equation}
that would lead to results in the same manner as we obtain from (\ref{F3}) and (\ref{F4}) above. For more information, see \cite{morgan1989global,morgan1990boundedness} in the case of constant diffusion coefficients.

\medskip
\noindent Our next main result is the following theorem.
\begin{theorem}\label{thm2}
	Assume \eqref{ellipticity}, \eqref{bounded_diffusion}, \eqref{bounded_drift}, \eqref{F1}, \eqref{F2}, \eqref{F5}, and \eqref{H1}, \eqref{H2}. Assume moreover that
	\begin{equation*}
		0\leq r < 1 + \frac{2}{n}.
	\end{equation*}
	Then for any non-negative, bounded initial data $u_0\in \LO{\infty}^m$, there exists a unique global weak solution to \eqref{Sys} with $u_i\in L^\infty_{\rm{loc}}(0,\infty;\LO{\infty})$ for all $i=1,\ldots, m$. Moreover, if $K_5< 0$ or $K_5 = K_6= 0$, then the solution is bounded uniformly in time, i.e. 
	\begin{equation*}
	{\rm{ess}\sup}_{t\geq 0}\|u_i(t)\|_{\LO{\infty}} < +\infty, \quad \forall i=1,\ldots, m.
	\end{equation*}
\end{theorem}

The following theorem gives a conditional result in the sense that if, by using the specific structure of the system, one can get better a-priori estimates, then the exponent $r$ in the intermediate sum condition \eqref{F4} can be enlarged.
\begin{theorem}\label{thm3}
	Assume \eqref{ellipticity}, \eqref{bounded_diffusion}, \eqref{bounded_drift}, \eqref{F1}, \eqref{F2}, \eqref{F4}, \eqref{F5}. Suppose that there exists either a constant $a\geq 1$ such that
	\begin{equation}\label{La}
		\|u_i\|_{L^{\infty}(0,T;\LO{a})} \leq \F(T), \quad \forall i=1,\ldots, m,
	\end{equation}
	or a constant $b\geq 1$ such that
	\begin{equation}\label{Lb}
		\|u_i\|_{L^{b}(0,T;L^b(\Omega))} \leq \F(T), \quad \forall i=1,\ldots, m,
	\end{equation}
	where $\F: [0,\infty) \to \R_+$ is a continuous function. Assume additionally that
	\begin{equation}\label{assump_r}
		0\leq r < \begin{cases}
			1 + 2a/n, &\text{ in case of } (\ref{La}),\\
			1 + 2b/(n+2), &\text{ in case of } (\ref{Lb}).
		\end{cases}
	\end{equation}
	Then for any non-negative, bounded initial data $u_0\in \LO{\infty}^m$, there exists a unique global weak solution to \eqref{Sys} with $u_i\in L^\infty_{\rm{loc}}(0,\infty;\LO{\infty})$ for all $i=1,\ldots, m$. Moreover, if $\sup_{T\geq 0}\F(T) < +\infty$, then the solution is bounded uniformly in time, i.e. 
	\begin{equation*}
	{\rm{ess}\sup}_{t\geq 0}\|u_i(t)\|_{\LO{\infty}} < +\infty, \quad \forall i=1,\ldots, m.
	\end{equation*}
\end{theorem}

\medskip
Our approach is sufficiently robust to extend to other boundary conditions such as Neumann or Robin type, or to certain quasilinear systems. The precise results are stated in the following theorems.
\begin{theorem}\label{thm3_1}
	Consider the system
	\begin{equation}\label{Sys_other_bc}
	\begin{cases}
		\pa_t u_i - \na\cdot(D_i(x,t)\na u_i) = F_i(x,t,u), &x\in\Omega, t>0,\\
		D_i(x,t)\na u_i(x,t)\cdot \nu(x) + \alpha_i u_i(x,t) = 0, &x\in\pa\Omega, t>0,\\
		u_{i}(x,0) = u_{i,0}(x), &x\in\Omega,
	\end{cases}
	\end{equation}
	where $\nu(x)$ is the outward unit normal vector on $\pa\Omega$, and $\alpha_i\geq 0$ for all $i=1,\ldots, m$.
	
	\medskip
	Assume \eqref{ellipticity}, \eqref{bounded_diffusion}, \eqref{F1}, \eqref{F2}, \eqref{F4} and \eqref{F5}. Moreover, assume either \eqref{F3} or \eqref{La} or \eqref{Lb} with
	\begin{equation*}
		0\leq r< \begin{cases}
			1+2/n, &\text{ in case of } \eqref{F3},\\
			1 + 2a/n, &\text{ in case of } \eqref{La},\\
			1 + 2b/(n+2), &\text{ in case of } \eqref{Lb}.
		\end{cases}
	\end{equation*}
	Then for any non-negative, bounded initial data $u_0\in \LO{\infty}^m$, there exists a unique global weak solution to \eqref{Sys_other_bc} (see Definition \ref{def_diff_bc}) with $u_i\in L^\infty_{\rm{loc}}(0,\infty;\LO{\infty})$ for all $i=1,\ldots, m$. Moreover, if $K_1< 0$ or $K_1= K_2 = 0$ in case of \eqref{F3}, or $\sup_{T\geq 0}\F(T)<+\infty$ in case of \eqref{La} or \eqref{Lb}, then the solution is bounded uniformly in time, i.e. 
	\begin{equation*}
	{\rm{ess}\sup}_{t\geq 0}\|u_i(t)\|_{\LO{\infty}} < +\infty, \quad \forall i=1,\ldots, m.
	\end{equation*}
\end{theorem}
\begin{remark}\hfill
	\begin{itemize}
	\item A direct corollary of Theorem \ref{thm3_1} is the global existence and boundedness of systems with Lipschitz nonlinearities. 
	\item 
	The uniform in time bound of solutions in Theorem \ref{thm3_1}, and previous theorems, is also important in numerical analysis. This can allow one to obtain convergence of the numerical scheme uniformly in time, see e.g. \cite{egger2018analysis}.
	\item Theorem \ref{thm3_1} can also be extended to the case of nonlinear boundary conditions, i.e.
	\begin{equation*}
		D_i(x,t)\na u_i(x,t)\cdot \nu(x) + \alpha_i u_i(x,t) = G_i(u), \quad x\in\pa\Omega,
	\end{equation*}
	where the nonlinearities $G_i$ also satisfy an intermediate sum condition. The details are left for the interested reader. We refer to \cite{morgan2021global,sharma2021global} for a related work dealing with constant diffusion coefficients.
	\end{itemize}
\end{remark}

\begin{theorem}\label{thm3_2}
	Consider the system
	\begin{equation}\label{Sys_quasilinear}
		\begin{cases}
		\pa_t u_i - \na\cdot(A_i(x,t,u)\na u_i) = F_i(x,t,u), &x\in\Omega, t>0,\\
		u_i(x,t) = 0, &x\in\pa\Omega, t>0,\\
		u_{i}(x,0) = u_{i,0}(x), &x\in\Omega,
		\end{cases}
	\end{equation}
	where, for each $i=1,\ldots, m$, $A_i:\Omega\times (0,\infty)\times \R_+^n \to \R^{n\times n}$ satisfies the following conditions:
	\begin{itemize}
		\item[(i)] For a.e. $(x,t)\in\Omega\times\R_+$, the map $\R^m\ni \omega \mapsto A_i(x,t,\omega)$ is continuous, for all $i=1,\ldots, m$.
		\item[(ii)] There exists $\widehat{\lam}>0$ such that 
		\begin{equation*}
			\widehat{\lam}|\xi|^2 \leq \xi^\top A_i(x,t,\omega)\xi \quad \forall (x,t,\omega)\in \Omega\times \R_+ \times \R_+^m, \; \forall \xi\in\R^n, \; \forall i=1,\ldots, m.
		\end{equation*}
		\item[(iii)] For each $T>0$,
		\begin{equation*}
		\underset{(x,t,\omega)\in\Omega\times(0,T)\times \R_+^m}{\text{\normalfont ess sup}}|A_i(x,t,\omega)| < +\infty.
		\end{equation*}
	\end{itemize}
	Assume \eqref{F1}, \eqref{F2}, \eqref{F4} and \eqref{F5}. Moreover, assume either \eqref{F3} or \eqref{La} or \eqref{Lb} with
	\begin{equation*}
	0\leq r< \begin{cases}
	1+2/n, &\text{ in case of } \eqref{F3},\\
	1 + 2a/n, &\text{ in case of } \eqref{La},\\
	1 + 2b/(n+2), &\text{ in case of } \eqref{Lb}.
	\end{cases}
	\end{equation*}
	Then for any non-negative, bounded initial data $u_0\in \LO{\infty}^m$, there exists a global weak solution to \eqref{Sys_quasilinear} with $u_i\in L^\infty_{\rm{loc}}(0,\infty;\LO{\infty})$ for all $i=1,\ldots, m$. Moreover, if $K_1< 0$ or $K_1= K_2 = 0$ in case of \eqref{F3}, or $\sup_{T\geq 0}\F(T)<+\infty$ in case of \eqref{La} or \eqref{Lb}, then this solution is bounded uniformly in time, i.e. 
	\begin{equation*}
	{\rm{ess}\sup}_{t\geq 0}\|u_i(t)\|_{\LO{\infty}} < +\infty, \quad \forall i=1,\ldots, m.
	\end{equation*}
	In addition, if the diffusion coefficients satisfy:
	\begin{itemize}
		\item[(iv)] for any $T>0$, the mapping $\overline{\Omega}\times [0,T]\ni (x,t)\mapsto A_i(\cdot,\cdot,\omega)$ is H\"older continuous in each component,
	\end{itemize}
	then the weak solution obtained above is unique.
\end{theorem}

\subsection{Structure of the paper}
In the next section, we provide the proofs of our main results. In subsection \ref{subs1}, we prove Theorem \ref{thm1} by first considering an approximate system where we regularize the nonlinearities to obtain global approximate weak solutions, then we derive uniform estimates, applying the key idea of the $L^p$-energy functions, and pass to the limit to obtain global existence and uniqueness of \eqref{Sys}. Subsection \ref{subs2} provides the proof of the generalized results in Theorems \ref{thm2} and \ref{thm3}. The last subsection \ref{subs3} gives the proof of Theorem \ref{thm4} concerning different types of boundary conditions or quasilinear systems. Section \ref{sec:applications} is devoted to the application of our results to an infectious disease without lifetime immunity. Finally, in Section \ref{sec:technical} we provide technical proofs concerning the construction of the $L^p$-energy functions that give rise to our uniform bounds.

\section{Proofs}
\subsection{Systems with control of mass: Proof of Theorem \ref{thm1}}\label{subs1}
To show the global existence of \eqref{Sys}, we start with a truncated system where the nonlinearities are regularized to be bounded. It is important that the regularization preserves the properties \eqref{F1}, \eqref{F2}, \eqref{F3}, \eqref{F4} and \eqref{F5}. The next crucial step is to derive uniform estimates for truncated systems where the bound is independent of the truncation. The main idea behind our approach is the development of an $L_{p}$-energy function for $p\ge2$. This allows us to obtain $L^p$-estimates for the approximate solutions for all $2\leq p < +\infty$, which is enough to conclude the boundedness thanks to the polynomial growth \eqref{F5}. The uniform-in-time bound is shown by examining the system on each cylinder $\Omega\times(\tau,\tau+2)$ for $\tau\in \mathbb N$. The final step of passing to the limit is straightforward because of our obtained a-priori estimates.


\subsubsection{Truncated systems}\label{subsec1}
For any $\eps>0$, we consider the truncated system: for all $i=1,\ldots, m$,
\begin{equation}\label{Sys_truncated}
\begin{cases}
	\pa_t \ue_i - \na\cdot(D_i\na \ue_i) + \na\cdot(B_i\ue_i) = F_i^\eps(\ue), &x\in\Omega, t>0,\\
	\ue_i(x,t) = 0, &x\in\pa\Omega, t>0,\\
	\ue_i(x,0) = \ue_{i,0}(x), &x\in\Omega,
\end{cases}
\end{equation}
where
\begin{equation}\label{Feps}
	F_i^\eps(\ue):=  F_i(\ue)\sbra{1+\eps\sumj|F_j(\ue)|}^{-1},
\end{equation}
and $\ue_{i,0}\in \LO{\infty}$ such that $\|\ue_{i,0} - u_{i,0}\|_{\LO{\infty}} \xrightarrow{\eps \to 0} 0$. Note that since $u_{i,0}\in \LO{\infty}$ one can take $\ue_{i,0} = u_{i,0}$ to obtain the global existence of global existence. The approximation $\ue_{i,0}$ we consider here can be relevant for numerical analysis, where the initial data is discretized and approximated by e.g. finite dimensional data.

\begin{lemma}\label{lem:existence_truncated}
	For any fixed $\eps>0$, there exists a global bounded, non-negative weak solution to \eqref{Sys_truncated}.
\end{lemma}
\begin{proof}
	\blue{Since for fixed $\eps>0$, the nonlinearities $F_i^\eps(u^\eps)$ are Lipschitz continuous and are bounded above by $1/\eps$ uniformly in $\ue$, the global existence of a unique weak solution follows from the standard Galerkin method, see e.g. \cite{LSU68}. We leave the details to the reader. It remains to show the non-negativity of $\ue$. Denote by $\ue_{i,+}:= \max\{\ue_i,0\}$ and $\ue_{i,-}:= \min\{\ue_i, 0\}$. We consider the auxiliary system of \eqref{Sys_truncated}
	\begin{equation*}
	\pa_t\ue_i - \na\cdot(D_i\na\ue_i) + \na\cdot(B_i\ue_i) = F_i^{\eps}(\ue_+)
	\end{equation*}
	where $\ue_+ = (\ue_{i,+})_{i=1,\ldots, m}$. Thanks to the uniqueness, it is sufficient to show the non-negativity for the solution of this system. By multiplying this auxiliary system by $\ue_{i,-}$ and using the quasi-positivity assumption \eqref{F2} (recall that this property also holds for $F_i^\eps$), we obtain
	\begin{equation*}
	\frac 12\intO|\ue_{i,-}|^2dx + \lam \intO|\na \ue_{i,-}|^2dx \leq \Gamma\intO|\ue_{i,-}||\na \ue_{i,-}| \leq \frac{\lam}{2}\intO |\na\ue_{i,-}|^2dx + C\int|\ue_{i,-}|^2dx.
	\end{equation*}
	By applying the Gronwall lemma, we get $\ue_{i,-} = 0$ a.e. in $Q_T$, since $\ue_{i,0,-} = 0$. This shows the desired non-negativity.}
\end{proof}

\subsubsection{Uniform-in-$\eps$ estimates}\label{subsec2}
In this subsection, we prove crucial uniform-in-$\eps$ estimates for the solution to the truncated system \eqref{Sys_truncated}. Moreover, we want to emphasize that \textit{all the constants in this subsection are {\it independent of $\eps$}}.

\medskip
We start off with the estimate in $L^\infty(0,T;L^1(\Omega))$.
\begin{lemma}\label{L1-bound}
	Assume \eqref{F1}, \eqref{F2} and \eqref{F3}. Then for any $T>0$, there exists a constant $M_T$ depending on $T, \Omega, \|u_{i,0}\|_{\LO{1}}$ and $c_1,\ldots, c_m, K_1, K_2$ in \eqref{F3} such that
	\begin{equation}\label{desired_L1}
		\sup_{t\in (0,T)}\|\ue_i(t)\|_{\LO{1}} \leq M_T \quad \forall i=1,\ldots, m.
	\end{equation}
\end{lemma}
\begin{proof}
	Formally, the desired $L^1$-bound can expected by first multiplying the equation of $u_i$ by $c_i$, summing the resultants, and then integrating the sum in $\Omega\times(0,t)$. Noticing that the solution is non-negative and vanishes on the boundary, we expect the flux $D_i\na \ue_i \cdot \nu \leq 0$. Unfortunately, since we are dealing with weak solutions, this formal procedure is not justified. On the other hand, $\varphi \equiv 1$ is not admissible as a test function since it does not belong to $L^2(0,T;H_0^1(\Omega))$. We therefore have to resort to a different strategy. Let $\delta>0$. We define the continuous (and piecewise smooth) function $h_\delta:\R\to\R$ as
	\begin{equation*}
		h_\delta(s) = \begin{cases}
			1 &\text{ if } s\geq \delta,\\
			s/\delta &\text{ if } 0<s<\delta,\\
			0 &\text{ if } s\leq 0.
		\end{cases}
	\end{equation*}
	We also use the following primitive of $h_\delta$ given by
	\begin{equation*}
		H_{\delta}(s) = 
		\begin{cases}
			s - \delta/2 & \text{ if } s\geq \delta,\\
			s^2/(2\delta) &\text{ if } 0<s<\delta,\\
			0 &\text{ if } s \leq 0.
		\end{cases}
	\end{equation*}
	Taking $h_{\delta}(\ue_i)\in L^2(0,T;H_0^1(\Omega))$ as a test function in \eqref{Sys_truncated} yields
	\begin{equation}\label{h1}
	\begin{aligned}
		\int_0^T\langle \pa_t \ue_i, h_\delta(\ue_i)\rangle dt + \intQT D_i\na \ue_i\cdot \na(h_\delta(\ue_i))dxdt - \intQT B_i\ue_i\cdot\na(h_\delta(\ue_i))dxdt\\
		= \intQT F_i^{\eps}(\ue)h_{\delta}(\ue_i)dxdt.
	\end{aligned}
	\end{equation}
	Since $\pa_t \ue_i \in L^2(0,T;H^{-1}(\Omega))$ and $h_{\delta}(\ue_i)\in L^2(0,T;H_0^1(\Omega))$ we have
	\begin{equation*}
		\int_0^T\langle \pa_t\ue_i, h_\delta(\ue_i)\rangle dt = \intO H_\delta(\ue_i(\cdot,T))dx - \intO H_{\delta}(\ue_i(\cdot,0))dx.
	\end{equation*}
	Due to $\na(h_\delta(\ue_i)) = \chi_{\{|\ue_i|\leq \delta\}}\delta^{-1}\na \ue_i$ and \eqref{ellipticity} we have
	\begin{equation*}
		\intQT D_i\na \ue_i \cdot \na(h_\delta(\ue_i))dxdt \geq \lam\intQT \chi_{\{|\ue_i|\leq \delta\}}\delta^{-1}|\na \ue_i|^2dxdt \geq 0,
	\end{equation*}
	and, using \eqref{bounded_drift},
	\begin{align*}
		\left|\intQT B_i\ue_i \cdot \na(h_\delta(\ue_i))dxdt\right| &\leq \delta^{-1}\intQT |B_i||\ue_i|\chi_{\{|\ue_i|\leq \delta\}}|\na \ue_i| dxdt\\
		&\leq \Gamma\intQT \chi_{|\ue_i|\leq\delta}|\na \ue_i|dxdt.
	\end{align*}
	Therefore, it follows from \eqref{h1} that
	\begin{equation}\label{h1_1}
		\intO H_\delta(\ue_i(\cdot,T))dx \leq \intO H_\delta(\ue_i(\cdot,0))dx + \Gamma\intQT \chi_{|\ue_i|\leq \delta}|\na \ue_i|dxdt + \intQT F_i^\eps(\ue)h_\delta(\ue_i)dxdt.
	\end{equation}
	Letting $\delta\to 0$, and using $\ue_i \in L^2(0,T;H_0^1(\Omega))$ we get
	\begin{equation}\label{h1_2}
		\limsup_{\delta\to 0}\intQT \chi_{|\ue_i|\leq \delta}|\na \ue_i|dxdt =0.
	\end{equation}
	{\color{black}
	We will show that
	\begin{equation*}
		\lim_{\delta \to 0}\intQT F_i^\eps(\ue)h_{\delta}(\ue_i)dxdt \leq \intQT F_i^\eps(\ue)dxdt.
	\end{equation*}
	Indeed, we write \begin{equation}\label{h1_3}
	F_i^{\eps}(\ue)h_\delta(\ue_i) = F_i^{\eps}(\ue)\mathbf{1}_{\{\ue_i > \delta\}} + \frac{1}{\delta}F_i^{\eps}(\ue)\ue_i \mathbf{1}_{\{0\leq \ue_i \leq \delta\}} =: \psi^{\delta}(x,t) + \varphi^{\delta}(x,t).
	\end{equation}
	It is clear that $\lim_{\delta\to0}\varphi^{\delta}(x,t) = 0$ for a.e. $(x,t)\in\Omega\times(0,T)$. Moreover, $|\varphi^{\delta}(x,t)| \leq |F_i^{\eps}(\ue(x,t))| \leq 1/\eps$ a.e. $(x,t)$ due to \eqref{Feps}. Thus, it holds $$\lim_{\delta\to 0}\int_0^T\int_{\Omega}\varphi^{\delta}(x,t)dxdt = 0.$$ For $(x,t)$ such that $\ue_i(x,t) > 0$, we have $\lim_{\delta \to 0}\psi^{\delta}(x,t) = F_i^{\eps}(\ue)$, while if $\ue_i(x,t) = 0$, then $\lim_{\delta \to 0}\psi^{\delta}(x,t) = 0 \leq F_i^{\eps}(\ue(x,t))$ thanks to \eqref{F2} and \eqref{Feps}. Hence, $$\lim_{\delta\to 0}\int_0^T\int_{\Omega}\psi^{\delta}(x,t)dxdt \leq \int_0^T\int_{\Omega}F_i^\eps(\ue)dxdt.$$
	Therefore, we can let $\delta \to 0$ in \eqref{h1_1} to obtain
	}
	\begin{equation*}
		\intO \ue_i(\cdot,T)dx \leq \intO \ue_i(\cdot,0)dx + \intQT F_i^\eps(\ue)dxdt.
	\end{equation*}
	Now, by using \eqref{F3}, we obtain
	\begin{align}\label{ee1}
		\sumi c_i\intO \ue_i(\cdot,T)dx  \leq \sumi c_i\intO \ue_i(\cdot,0)dx + \intQT \bra{K_1 \sumi \ue_i + K_2}dxdt.
	\end{align}
	The Gronwall lemma gives the desired $L^1$-estimate.
\end{proof}

\noindent\textbf{Building $L^p$-energy functions}: To this end, we write $\mathbb{Z}_{+}^{m}$ as the set of all $m$-tuples of non-negative integers. Addition and
	scalar multiplication by non-negative integers of elements in $\mathbb{Z}_{+}^{m}$
	is understood in the usual manner. If $\beta=(\beta_{1},...,\beta_{m})\in \mathbb{Z}_{+}^{m}$ and $p\in \mathbb N\cup \{0\}$,
	then we define $\beta^{p}=((\beta_{1})^{p},...,(\beta_{m})^{p})$.
	Also, if $\alpha=(\alpha_{1},...,\alpha_{m})\in  \mathbb{Z}_{+}^{m}$, then
	we define $|\alpha|=\sum_{i=1}^{m}\alpha_{i}$. Finally, if $z=(z_{1},...,z_{m})\in \mathbb{R}_{+}^{m}$
	and $\alpha=(\alpha_{1},...,\alpha_{m})\in \mathbb{Z}_{+}^{m}$, then we define
	$z^{\alpha}=z_{1}^{\alpha_{1}}\cdot...\cdot z_{m}^{\alpha_{m}}$,
	where we interpret $0^{0}$ to be $1$. For $p\in \mathbb N\cup \{0\}$, we build our $L^p$-energy function of the form
	\begin{equation}\label{Lp}
	\L_p[\ue](t) = \intO \H_p[\ue](t)dx
	\end{equation}
	where
	\begin{equation}\label{Hp}
	\H_p[\ue](t) = \sum_{\beta\in \mathbb Z_+^{m}, |\beta| = p}\begin{pmatrix}
	p\\ \beta\end{pmatrix}\theta^{\beta^2}\ue(t)^{\beta},
	\end{equation}
	with
	\begin{equation}\label{our-def}
	\begin{pmatrix}
	p\\ \beta\end{pmatrix}=\frac{p!}{\beta_1!\cdots\beta_{m}!},
	\end{equation}
	and $\theta= (\theta_1,\ldots, \theta_{m})$ where $\theta_1,...,\theta_{m}$ are positive real numbers which will be determined later. For convenience, hereafter we drop the subscript $\beta\in \mathbb Z_+^{m}$ in the sum as it should be clear. \blue{Note that if $\theta = (1,\ldots, 1)$ in \eqref{Hp}, we have
	\begin{equation*}
		\H_p[\ue](t) = \sum_{|\beta| = p}\begin{pmatrix}p\\\beta \end{pmatrix}\ue(t)^{\beta} = \bra{\sumi \ue_i(t)}^p.
	\end{equation*}
	Therefore, for a given $2\le p\in \mathbb N$, $\H_p[\ue]$ is a generalization of a multinomial expansion of degree $p$ in $\ue$. For $p=0,1,2$, one can write these functions explicitly as
	$$\H_0[\ue](t)=1\text{ and }\H_1[\ue](t)=\sum_{j=1}^{m}\theta_j\ue_j(t)$$
	and
	\begin{equation*}
		\H_2[\ue](t) = \sumi \theta_i^4\ue_i(t)^2 + 2\sum_{i=1}^{m-1}\sum_{j=i+1}^m\theta_i\theta_j\ue_i(t)\ue_j(t).
	\end{equation*}
	Thanks to the non-negativity of the solution, we have
	$$\L_p[\ue](t) \sim \sumi \|\ue_i(t)\|_{\LO{p}}^p.$$
	Therefore, we are going to use $\L_p[\ue](t)$ to obtain a priori estimates on $\ue$ for each $2\le p\in \mathbb N$. We will need two technical lemmas, concerning the derivative in time of $\H_p$ and integration by parts, whose proofs are postponed to Section \ref{sec:technical} in order not to interrupt the train of thought.}

	\medskip
	Next, we need the following functional inequality, which was proved in \cite{morgan2021global}. 
\begin{lemma}\label{GN}\cite[Lemma 2.3]{morgan2021global}
	Suppose $\Omega\subset \mathbb{R}^n$ such that the Gagliardo-Nirenberg inequality is satisfied and basic trace theorems apply (for instance $\Omega$ has Lipschitz boundary). Let $a\ge 1$, $p\ge 2a$ and $w:\overline\Omega\to\mathbb{R}_+$ such that $w^{p/2}\in H^1(\Omega)$ and there exists $K\ge 0$ such that $\|w\|_{a,\Omega}\le K$. If $0\le s < 2a/n$ and $\kappa>0$, then there exists $C_{\kappa}\ge 0$ (depending on $p,\kappa,r,a,\Omega$, but independent of $w$) such that
	\begin{align}\label{youbetcha}
	\int_\Omega w^{p+s}dx\le \kappa\int_\Omega \left(w^{p-2}|\nabla w|^2+w^p\right) dx+C_{\kappa}.
	\end{align}
\end{lemma}
The following lemma shows a consequence of the intermediate sum condition \eqref{F4} which is crucial to the construction of the $L^p$-energy function.
\begin{lemma}\label{auxiliary_functions}
	Assume \eqref{F4}. Then there exist componentwise increasing functions $g_i: \R^{m-i}\to\R_+$ for $i=1,\ldots, m-1$ such that: if $\theta = (\theta_1,\ldots, \theta_m)\in (0,\infty)^m$ satisfies $\theta_m>0$ and $\theta_{i} \ge g_i(\theta_{i+1},\ldots, \theta_m)$ for all $i=1,\ldots, m-1$, then 
	\begin{equation*}
	\sumi \theta_i F_i^{\eps}(x,t,\ue) \leq K_{\theta}\bra{1+\sumi (\ue_i)^{r}} \quad \forall (x,t,\ue)\in\Omega\times\R_+\times \R_+^m
	\end{equation*}
	for some constant $K_{\theta}$ depending on $\theta$, $g_i$, and $K_3$ in \eqref{F4}.
\end{lemma}
\begin{proof}
	The proof follows exactly from \cite[Lemma 2.2]{morgan2021global} with the observation
	\begin{equation*}
		\sumi \theta_iF_i^{\eps}(x,t,\ue) \leq \max\left\{0;\sumi \theta_iF_i(x,t,\ue)\right\} \quad \forall (x,t,\ue)\in \Omega\times\R_+\times\R_+^m.
	\end{equation*}
\end{proof}

We are now ready to use the $L^p$-energy functions built in \eqref{Lp}-\eqref{Hp} to obtain the $L^p$-estimates of $\ue$.
\begin{lemma}\label{Lp-bound}
	Assume \eqref{F1}, \eqref{F2}, \eqref{F3} and \eqref{F4}. Then for any $1\leq p < \infty$ and any $T>0$, there exists a constant $C_{T,p}$ depending on $T$, $p$ and other parameters such that
	\begin{equation*}
		\sup_{t\in (0,T)}\|\ue_i(t)\|_{\LO{p}} \leq C_{T,p} \quad \forall i=1,\ldots, m.
	\end{equation*}
\end{lemma}
\begin{proof}
	Let $\ue$ solve (\ref{Sys_truncated}), and $\L_p(t):= \L_p[\ue](t)$ be defined in \eqref{Lp}.
	Then 
	\begin{align*}
	\L_p'(t)&=\int_{\Omega}\sum_{|\beta|=p-1}\left(\begin{array}{c}
	p\\
	\beta
	\end{array}\right)\theta^{\beta^{2}}\ue(x,t)^{\beta}\sum_{k=1}^{m}\theta_{k}^{2\beta_{k}+1}\frac{\partial}{\partial t}\ue_{k}(x,t)dx\\
	&=\int_{\Omega}\sum_{|\beta|=p-1}\left(\begin{array}{c}
	p\\
	\beta
	\end{array}\right)\theta^{\beta^{2}}\ue(x,t)^{\beta}\sum_{k=1}^{m}\theta_{k}^{2\beta_{k}+1}\\
	&\qquad\qquad \times \biggl[\na\cdot(D_k(x,t)\na \ue_k(x,t)) - \na\cdot(B_k\ue_k)+F_{k}(\ue(x,t))\biggr]dx.
	\end{align*}
	If we apply Lemma \ref{Hp-lem8} and integration by parts, we have
	\[
	\int_{\Omega}\sum_{|\beta|=p-1}\left(\begin{array}{c}
	p\\
	\beta
	\end{array}\right)\theta^{\beta^{2}}\ue(x,t)^{\beta}\sum_{k=1}^{m}\theta_{k}^{2\beta_{k}+1}\na\cdot(D_k(x,t)\na \ue_k(x,t))dx=I,
	\]
	where 
	\[
	I=-\int_{\Omega}\sum_{|\beta|=p-2}\left(\begin{array}{c}
	p\\
	\beta
	\end{array}\right)\theta^{\beta^{2}}\ue(x,t)^{\beta}\sum_{k=1}^{m}\sum_{l=1}^{m}C_{k,r}(\beta)\left(D_k\nabla \ue_k\right)\cdot\nabla \ue_l dx
	\]
	with
	\[
	C_{k,l}(\beta)=\begin{cases}
	\begin{array}{cc}
	\theta_{k}^{2\beta_{k}+1}\theta_{l}^{2\beta_{l}+1}, & k\ne l,\\
	\theta_{k}^{4\beta_{k}+4}, & k=l.
	\end{array}\end{cases}
	\]
	For a given $\beta$ with $|\beta|=p-2$, create an $mn\times mn$
	matrix $B(\beta)$ made up of $m^{2}$ blocks $B_{k,l}(\beta)$, each
	of size $n\times n$, where 
	\[
	B_{k,l}(\beta)=\frac{1}{2}C_{k,l}(\beta)\left(D_{k}+D_{l}\right).
	\]
	Note that for each $k=1,...,m$, 
	\[
	B_{k,k}(\beta)=\theta_{k}^{4\beta_{k}+4}D_{k}.
	\]
	Also, 
	\[
	I=-\int_{\Omega}\sum_{|\beta|=p-2}\left(\begin{array}{c}
	p\\
	\beta
	\end{array}\right)\theta^{\beta^{2}}\ue(x,t)^{\beta}\nabla \ue(x,t)^{T}B(\beta)\nabla \ue(x,t)dx,
	\]
	where $\nabla \ue(x,t)$ is a column vector of size $mn\times1$, and
	for $j=1,...,m$, entries $n(j-1)+1\mbox{ to \ensuremath{nj}}$ of
	$\nabla \ue(x,t)$ are $\nabla \ue_{j}(x,t)$. 
	We claim that if all of the entries in $\theta$ are sufficiently large,
	then $B(\beta)$ is positive definite. In fact, it is a simple matter to show
	it is positive definite if and only if the $mn\times mn$ matrix $\widetilde{B}(\beta)$
	made up of $n\times n$ blocks 
	\[
	\widetilde{B}_{k,l}(\beta)=\begin{cases}
	\begin{array}{cc}
	\theta_{k}^{2}D_k, & k=l,\\
	\frac{1}{2}\left(D_{k}+D_{l}\right), & k\ne l,
	\end{array}\end{cases}
	\]
	is positive definite. However, if we recall the uniform positive definiteness
	of the matrices $D_k$, we can show that if $\theta_{i}$
	is sufficient large for each $i$, then we have what we need. Consequently, returning to above, we can show there exists $\alpha_p>0$ so that
	\begin{align}\label{nearly-there}
	\L_p'(t)&+\alpha_p\sum_{k=1}^m \int_\Omega |\nabla (\ue_k)^{p/2}(x,t)|^2 dx \nonumber\\
	&\le\int_{\Omega}\sum_{|\beta|=p-1}\left(\begin{array}{c}
	p\\
	\beta
	\end{array}\right)\theta^{\beta^{2}}\ue(x,t)^{\beta}\sum_{k=1}^{m}\theta_{k}^{2\beta_{k}+1}\biggl[-\na\cdot(B_k \ue_k) + F_{k}^{\eps}(\ue(x,t))\biggr]dx.
	\end{align}
	For the first term on the right hand side of \eqref{nearly-there} we can use integration by parts and H\"older's inequality to obtain
	\begin{equation}\label{e1}
	\begin{aligned}
		&-\int_{\Omega}\sum_{|\beta|=p-1}\left(\begin{array}{c}
		p\\
		\beta
		\end{array}\right)\theta^{\beta^{2}}\ue(x,t)^{\beta}\sum_{k=1}^{m}\theta_{k}^{2\beta_{k}+1}\na\cdot(B_k\ue_k)\\
		&\leq \frac{\alpha_p}{2}\sumk\intO |\na (\ue_k)^{p/2}(x,t)|^2dx + C_{\theta,p}\sumk \intO |\ue_k|^pdx.
	\end{aligned}
	\end{equation}
	We now look closely at the second term on the right hand side of (\ref{nearly-there}), and in particular the term
	\[
	\sum_{k=1}^{m}\theta_{k}^{2\beta_{k}+1}F_{k}^{\eps}(\ue).
	\]
	Note that from (\ref{F4}) and Lemma \ref{auxiliary_functions}, there exist componentwise increasing functions $g_i:\mathbb{R}^{m-i}\to\mathbb{R}_+$ for $i=1,...,m-1$ so that if $\gamma_m>0$ and $\gamma_i\ge g_i(\gamma_{i+1},...,\gamma_m)$ for $i=1,...,m-1$ then there exists $K_\gamma>0$ so that
	\[
	\sum_{k=1}^m \gamma_k F_k^{\eps}(x,t,\ue)\le K_\gamma\bra{1+\sumi (\ue_i)^{r}}\text{ for all }(x,t,\ue)\in \Omega\times\mathbb{R}_+\times\mathbb{R}_+^m.
	\]
	So, we choose $\theta$ so that its components are sufficiently large that the previous positive definiteness condition is satisfied, and 
	\[
	\theta_i\ge g_i(\theta_{i+1}^{2p-1},...,\theta_m^{2p-1})\text{ for }i=1,...,m-1,
	\] 
	where $g_i$ are functions constructed in Lemma \ref{auxiliary_functions}. Then there exists $K_{\tilde\theta}$ so that for all $\beta\in\mathbb{Z}_+$ with $|\beta|=p-1$, we have
	\[
	\sum_{k=1}^{m}\theta_{k}^{2\beta_{k}+1}F_{k}^{\eps}(\ue(x,t))\le K_{\tilde\theta}\bra{1+\sumi (\ue_i)^r}\text{ for all }(x,t,\ue)\in \Omega\times\mathbb{R}_+\times\mathbb{R}_+^m.
	\]
	It follows from this and \eqref{e1} that there exists $C_p>0$ so that (\ref{nearly-there}) implies
	\begin{align}\label{closer-still}
	\L_p'(t)+\frac{\alpha_p}{2}\sum_{k=1}^m \int_\Omega |\nabla (\ue_k)^{p/2}(x,t)|^2 dx
	&\le C_p\int_{\Omega}\sum_{k=1}^m\left(\ue_k(x,t)^{p-1+r} + \ue_k(x,t)^p+1\right) dx.
	\end{align}
	By adding both sides with $\frac{\alpha_p}{2}\sumk\intO (\ue_k)^pdx$ and using
	\begin{equation*}
		(\ue_k)^{p} \leq C\sbra{1+(\ue_k)^{p-1+r}} \quad \text{due to} \quad r\geq 1
	\end{equation*}
	we get
	\begin{equation}\label{e3}
		\L_p'(t)+\frac{\alpha_p}{2}\sumk\intO\bra{|\na (\ue_k)^{p/2}|^2 + |\ue_k|^p}dx \leq C_p\bra{1+\sumk\intO (\ue_k)^{p-1+r}dx}.
	\end{equation}
	Applying Lemma \ref{GN} with $a = 1$ and $1\leq r < 1 + 2/n$ to the right hand side above, implies there exists $C_{T,p}>0$, $C>0$ and $\delta>0$ so that
	\begin{align}\label{muchcloser-still}
	\L_p'(t) + C\int_{\Omega}\sum_{k=1}^m \ue_k(x,t)^p dx\le C_{T,p},
	\end{align}
	which implies
	\begin{align}\label{muchcloser-still_1}
	\L_p'(t) + \delta \L_p(t)\le C_{T,p},
	\end{align}
	for some $\delta>0$. Clearly, (\ref{muchcloser-still_1}) allows us to obtain the estimates
	\begin{equation*}
		\sup_{t\geq 0}\L_p(t) \le C_{T,p},
	\end{equation*}
	and these in turn allow us to obtain estimates for 
	\begin{equation*}
		\sup_{t\geq 0}\|\ue_k(t)\|_{\LO{p}} \leq C_{T,p} \quad \forall k=1,\ldots, m.
	\end{equation*}
\end{proof}
Combining the $L^p$-estimates and the polynomial growth of the nonlinearities in \eqref{F5} we ultimately obtain the boundedness in $L^\infty$.

\begin{proposition}\label{pro1}
	Assume \eqref{ellipticity}, \eqref{bounded_diffusion}, \eqref{bounded_drift}, \eqref{F1}, \eqref{F2}, \eqref{F3}, \eqref{F4} and \eqref{F5}. Then for any $T>0$ the solution to the truncated system \eqref{Sys_truncated} is bounded in $L^\infty$ locally in time, i.e.
	\begin{equation*}
		\|\ue_i\|_{\LQ{\infty}} \leq C_T \quad \forall i=1,\ldots,m.
	\end{equation*}
	for some constant $C_T$ depending on $T$ and independent of $\eps>0$.
\end{proposition}
\begin{proof}
	From \eqref{F5},
	\begin{equation*}
		F_i^{\eps}(\ue)\leq F_i(\ue) \leq G_i(\ue):= K_4\sbra{1+\sumi (\ue_i)^{\ell}}.
	\end{equation*}
	Lemma \ref{Lp-bound} implies that for any $1\leq p <\infty$ a constant $C_p>0$ exists depending on $p$ (but not on $T$) such that
	\begin{equation*} 
		\|G_i(\ue)\|_{\LQ{p}} \leq C_p \quad \forall i=1,\ldots, m.
	\end{equation*}
	By choosing $p> \frac{n+2}{2}$ and using the regularization of the parabolic operator $\pa_tv - \na\cdot(D_i\na v) + \na\cdot(B_iv)$ (see e.g. \cite{LSU68} or \cite[Proposition 3.1]{nittka2014inhomogeneous}) we obtain
	\begin{equation*}
		\|u_i\|_{\LQ{\infty}}\le C_T \quad \forall i=1,\ldots, m.
	\end{equation*}
	\blue{Indeed, this follows exactly the same as in the proof of Proposition 3.1 in \cite{nittka2014inhomogeneous} with a slight modification concerning (A.15) therein. More precisely, with $f(s) = F_i^\eps(\ue(s))$ and $u^{(k)}(s) = (\ue_i(s) - k)_+\ge 0$ in  \cite[Eq. (A.15)]{nittka2014inhomogeneous}, we can estimate
	\begin{equation*}
		\intO f(s)u^{(k)}(s)dx = \intO F_i^{\eps}(\ue(s))(\ue_i(s)-k)_+dx \leq \intO G_i(\ue(s))(\ue_i(s)-k)_+dx
	\end{equation*}
	and the rest of the proof of \cite[Proposition 3.1]{nittka2014inhomogeneous} goes through.}
	This completes the proof of Proposition \ref{pro1}.
\end{proof}
\subsubsection{Passing to the limit - Global existence}\label{subsec3}
	In this Subsection, we pass to the limit $\eps\to 0$ in \eqref{Sys_truncated} using the uniform estimates in Subsection \ref{subsec2}.
	\begin{proof}[Proof of Theorem \ref{thm1}: Global existence]
		From Subsection \ref{subsec2} we have the following uniform-in-time bound
		\begin{equation*}
			\|\ue_i\|_{\LQ{\infty}} \leq C_T \quad \forall i=1,\ldots, m.
		\end{equation*}
		Due to the polynomial growth \eqref{F5}, it follows that
		\begin{equation*}
			\|F_i(\ue)\|_{\LQ{\infty}} \leq C_T \quad \forall i=1,\ldots, m.
		\end{equation*}
		By multiplying \eqref{Sys_truncated} by $\ue_i$ then integrating on $Q_T$ gives
		\begin{align*}
			\frac 12 \|\ue_i(T)\|_{\LO{2}}^2 + \intQT (D_i\na\ue_i)\cdot \na \ue_i dxdt\\
			= \frac 12 \|\ue_{i,0}\|_{\LO{2}}^2 + \intQT \ue_iB_i\cdot \na \ue_i dxdt + \intQT F_i^{\eps}(\ue)\ue_i dxdt.
		\end{align*}
		The ellipticity \eqref{ellipticity} and H\"older's inequality give
		\begin{equation*}
			\intQT (D_i\na\ue_i)\cdot \na\ue_i dxdt \geq \lam\|\ue_i\|_{\LQ{2}}^2
		\end{equation*}
		and
		\begin{equation*}
			\intQT \ue_iB_i\cdot \na\ue_idxdt \leq \frac{\lam}{2}\|\na \ue_i\|_{\LQ{2}}^2 + \frac{\Gamma^2}{2\lam}\|\ue_i\|_{\LQ{2}}^2,
		\end{equation*}
		and consequently 
		\begin{equation*}
			\|\na\ue_i\|_{\LQ{2}}^2 \leq C_T \quad \forall i=1,\ldots, m.
		\end{equation*}
		Thus,
		\begin{equation*}
			\{\ue_i\}_{\eps>0} \quad \text{ is bounded uniformly in $\eps$ in } \quad \LQ{\infty}\cap L^2(0,T;H_0^1(\Omega)).
		\end{equation*}
		From this, it follows easily that
		\begin{equation*}
			\{\pa_t\ue_i\}_{\eps>0} \quad \text{ is bounded uniformly in $\eps$ in } \quad L^2(0,T;H^{-1}(\Omega)).
		\end{equation*}
		The classical Aubin-Lions lemma gives the strong convergence (up to a subsequence)
		\begin{equation*}
			\ue_i \xrightarrow{\eps\to 0} u_i \quad \text{ strongly in } \quad \LQ{2}.
		\end{equation*}
		Consequently, for any $1\leq p<\infty$,
		\begin{equation*}
			\ue_i \xrightarrow{\eps\to 0} u_i \quad \text{ strongly in } \quad \LQ{p},
		\end{equation*}
		thanks to the uniform $L^\infty$-bound of $\ue_i$. This is enough to pass to the limit in the weak formulation of \eqref{Sys_truncated}
		\begin{align*}
		-\intO \varphi(\cdot,0)\ue_{i,0}dx &-\intQT \ue_i \pa_t\varphi dxdt + \intQT D_i\na\ue_i \cdot \na\varphi dxdt\\
		&-\intQT \ue_i B_i\cdot \na\varphi dxdt = \intQT F_i^{\eps}(\ue)\varphi dxdt,
		\end{align*}
		to obtain that $u = (u_1,\ldots, u_m)$ is a global weak solution to \eqref{Sys} and additionally
		\begin{equation*}
			\|u_i\|_{\LQ{\infty}} \le C_T \quad \forall i=1,\ldots, m.
		\end{equation*}
	\end{proof}
\subsubsection{Uniform-in-time estimates}\label{subsec4}

\begin{lemma}\label{L1-bound-uniform}
	Assume \eqref{F1}, \eqref{F2} and \eqref{F3} with either $K_1 < 0$ or $K_1 = K_2 = 0$. Then, there exists a constant $M$ independent of time such that
	\begin{equation}\label{desire-L1}
		\sup_{t\geq 0}\|u_i(t)\|_{\LO{1}} \leq M \quad \forall i=1,\ldots, m.
	\end{equation}
\end{lemma}
\begin{proof}
	{\color{black}
	By the same arguments in Lemma \ref{L1-bound}, we have, similarly to \eqref{ee1},
	\begin{equation*}
		\sumi c_i\intO u_i(\cdot,t)dx \leq \sumi c_i\intO u_i(\cdot,s)dx + \int_s^t\intO K_1\sumi u_i(\cdot,r)dxdr + K_2|\Omega|(t-s)
	\end{equation*}
	for all $t> s\geq 0$. If $K_1 = K_2 = 0$, the desired bound \eqref{desire-L1} follows immediately. If $K_1<0$, we get for some constant $\sigma>0$, and for all $t>s \ge 0$
	\begin{equation}\label{d1}
		\sumi c_i\intO u_i(t)dx + \sigma \int_s^t\bra{\sumi c_i\intO u_i(r)dx}dr\leq \sumi c_i\intO u_{i}(s)dx + K_2|\Omega|(t-s).
	\end{equation}
	Define $\psi(t) = \sumi c_i\intO u_i(\cdot,t)dx$ and $\varphi(s) = \int_s^t\bra{\sumi c_i\intO u_i(\cdot,r)dx}dr$. It follows from \eqref{d1} that
	\begin{equation*}
		\varphi'(s) = -\sumi c_i\intO u_i(\cdot,s)dx \leq -\psi(t) - \sigma \varphi(s) + K_2|\Omega|(t-s),
	\end{equation*}
	which leads to
	\begin{equation*}
		(e^{\sigma s}\varphi(s))' + e^{\sigma s}\psi(t) \leq K_2|\Omega|e^{\sigma s}(t-s).
	\end{equation*}
	Integrating with respect to $s$ on $(0,t)$, and using $\varphi(t) = 0$, we have
	\begin{equation*}
		-\varphi(0) + \psi(t)\frac{e^{\sigma t} - 1}{\sigma} \leq \frac{K_2|\Omega|}{\sigma}\bra{-t + \frac{e^{\sigma t}-1}{\sigma}}.
	\end{equation*}
	Since $\sigma \varphi(0) - K_2|\Omega|t \leq \sumi c_i\intO u_{i0}(x)dx$ (see \eqref{d1}) it follows that
	\begin{equation*}
		\psi(t) \leq (e^{\sigma t}-1)^{-1}\sumi c_i\intO u_{i,0}(x)dx + K_2|\Omega|\sigma^{-1},
	\end{equation*}
	which finishes the proof of Lemma \ref{L1-bound-uniform}.
	}
\end{proof}
With the $L^1$-bound (uniformly in time), we are ready to show the uniform-in-time boundedness of the solution.
\begin{proof}[Proof of Theorem \ref{thm1}: Uniform-in-time boundedness]
	We first show that for any $1\leq p<\infty$, there exists a constant $C_p>0$ such that
	\begin{equation}\label{e4}
		\sup_{t\geq 0}\|u_i(t)\|_{\LO{p}} \leq C_p \quad \forall i=1,\ldots, m.
	\end{equation}
	By using the $L^p$-energy function $\L_p[u]$ defined in \eqref{Lp} and the computations similar to Lemma \ref{Lp-bound}, we obtain \eqref{e3} which we recall here
	\begin{equation*}
		\L_p' + \frac{\alpha_p}{2}\sumk\intO \bra{|\na u_k|^{p/2} + |u_k|^p}dx \leq C_p\bra{1+\sumk \intO u_k^{p-1+r}dx}.
	\end{equation*}
	We now apply Lemma \ref{GN} to the right hand side, bearing in mind that the $L^1$-bound is uniform in time (thanks to Lemma \ref{L1-bound-uniform}), to get
	\begin{equation*}
		\L_p'(t) + C\intO \sumk u_k(x,t)^pdx \leq C
	\end{equation*}
	which leads to
	\begin{equation*}
		\L_p'(t) + \sigma \L_p(t) \leq C
	\end{equation*}
	for some constant $\sigma>0$. Gronwall's lemma yields
	\begin{equation*}
		\sup_{t\geq 0}\L_p(t) \leq C,
	\end{equation*}
	which implies \eqref{e4}. To finally see that the solution is bounded uniformly in time in sup norm, we use a smooth time-truncated function $\psi: \R \to [0,1]$ with $\psi(s) = 0$ for $s\leq 0$ and $\psi(s) = 1$ for $s\geq 1$, $0\leq \psi' \leq C$ and its shifted version $\psi_\tau(\cdot) = \psi(\cdot-\tau)$ for any $\tau\in \mathbb N$. Let $\tau\in\mathbb N$ be arbitrary. It is straightforward to show that since $u = (u_i)_{i=1,\ldots,m}$ is a weak solution to \eqref{Sys}, the function $\psi_\tau u = (\psi_\tau u_i)_{i=1,\ldots,m}$ is a weak solution to the following
	\begin{equation*}
	\begin{cases}
		\pa_t(\psi_\tau u_i) - \na\cdot(D_i\na(\psi_\tau u_i)) + \na\cdot(B_i(\psi_\tau u_i)) = \psi_\tau' u_i + \psi_\tau F_i(x,t,u), &x\in\Omega, t\in (\tau,\tau+2)\\
		(\psi_\tau u_i)(x,t) =0, &x\in\pa\Omega, t\in (\tau,\tau+2),\\
		(\psi_\tau u_i)(x,\tau) = 0, &x\in\Omega.
	\end{cases}
	\end{equation*}
	Thanks to \eqref{e4} and the polynomial growth \eqref{F5}, we have 
	\begin{equation*}
		\psi_\tau' u_i + \psi_\tau F_i(x,t,u) \leq G_i(x,t,u):= C\bra{1+ \sumk u_k^{\ell}}.
	\end{equation*}
	Thanks to \eqref{e4} for any $1\leq p<\infty$ a constant $C_p>0$ exists such that
	\begin{equation*}
		\|G_i(x,t,u)\|_{L^{p}(\Omega\times(\tau,\tau+2))} \leq C_p \quad \forall i=1,\ldots, m.
	\end{equation*}
	Therefore, by the smooth effect of parabolic operator $\pa_tv - \na\cdot(D_i\na v) + \na\cdot(B_i v )$ (see e.g. \cite[Proposition 3.1]{nittka2014inhomogeneous} or \cite{LSU68}), similar to Proposition \ref{pro1},  we get
	\begin{equation*}
		\|\psi_\tau u_i\|_{L^{\infty}(\Omega\times(\tau,\tau+2))} \leq C \quad \forall i=1,\ldots, m,
	\end{equation*}
	where $C$ is a constant {\it independent of $\tau\in \mathbb N$}. Thanks to $\psi_\tau \geq 0$ and $\psi|_{(\tau+1,\tau+2)} \equiv 1$, we obtain finally the uniform-in-time bound
	\begin{equation*}
		\sup_{t\geq 0}\|u_i(t)\|_{\LO{\infty}} \leq C \quad \forall i=1,\ldots, m.
	\end{equation*}
\end{proof}


\subsection{Generalizations: Proof of Theorems \ref{thm2} and \ref{thm3}}\label{subs2}
\begin{proof}[Proof of Theorem \ref{thm2}] We give a formal proof as its rigor can be easily obtained through approximation. Since $h_k(\cdot)$ is convex, we have
	\begin{align*}
		\na\cdot(D_i\na(h_i(u_i))) &= \na\cdot(D_ih_i'(u_i)\na u_i)\\
		&= h'(u_i)\na\cdot(D_i\na u_i) + h_i''(u_i)D_i|\na u_i|^2\\
		&\geq h'(u_i)\na \cdot(D_i\na u_i).
	\end{align*}
	Therefore, by defining $v_i:= h_i(u_i)\geq 0$ we have
	\begin{equation}\label{ee2}
		\pa_t v_i - \na\cdot(D_i\na v_i) \leq G_i(x,t,u):= h_i'(u_i)F_i(x,t,u), \quad x\in\Omega, \; t>0,
	\end{equation}
	with initial data $v_i(x,0) = h_i(u_{i,0}(x))$ and homogeneous Dirichlet boundary condition $v_i(x,t) = 0$ for $x\in\pa\Omega$ and $t>0$. Thanks to \eqref{H1} and  we have
	\begin{equation*}
		\sumi G_i(x,t,u) \leq K_5\sumi v_i + K_6
	\end{equation*}
	and
	\begin{equation*}
		A\begin{pmatrix}G_i(x,t,u)\\ \cdots\\ G_m(x,t,u)\end{pmatrix} \leq K_7\overrightarrow{1}\bra{\sumi v_i + 1}^r.
	\end{equation*}
	We can now reapply the methods in the proof of Theorem \ref{thm1}, keeping in mind that though $v_i$ satisfies \eqref{ee2} with an inequality all calculations are still available thanks to its non-negativity, to obtain $v_i \in L^\infty_{\text{loc}}(0,\infty;\LO{\infty})$, and in case $K_5<0$ or $K_5 = K_6 = 0$ in \eqref{H1},
	\begin{equation*}
		\text{ess}\sup_{t\geq 0}\|v_i(t)\|_{\LO{\infty}} <+\infty, \quad \forall i=1,\ldots, m.
	\end{equation*}
	Thanks to the assumption \eqref{H1}, the global existence and boundedness of \eqref{Sys} follow immediately.
\end{proof}
\begin{proof}[Proof of Theorem \ref{thm3}]
	For the global existence, we only need to show that for any $1\leq p<\infty$ and any $T>0$, there exists $C_{T,p}>0$ such that
	\begin{equation}\label{e6}
		\|u_i\|_{\LQ{p}} \leq C_{T,p} \quad \forall i=1,\ldots, m.
	\end{equation}
	The rest follows exactly as in Proposition \ref{pro1}. To show \eqref{e6}, we utilize the $L^p$-energy functions $\L_p(t)$ constructed in Lemma \ref{Lp-bound}. Repeat the arguments in the proof of Lemma \ref{Lp-bound} until \eqref{e3} we end up with
	\begin{equation}\label{e7}
		\L_p'(t)+ \frac{\alpha_p}{2}\sumk \intO\bra{|\na u_k^{p/2}|^2 + |u_k|^p}dx \leq C_p\bra{1+\sumk \intO u_k^{p-1+r}dx}.
	\end{equation}
	
	\medskip
	\noindent\underline{\it In case \eqref{La} holds}, we apply Lemma \ref{GN} to estimate
	\begin{equation*}
		\intO u_{k}^{p-1+r}dx \leq \frac{\alpha_p}{2}\intO \bra{|\na u_k^{p/2}|^2 + |u_k|^p}dx + C_{T}
	\end{equation*}
	where $C_{T}$ depends on $\F(T)$. Inserting this into \eqref{e7} yields
	\begin{equation*}
		\L_p'(t) + \delta \L_p(t) \leq C_{p,T},
	\end{equation*}
	which implies \eqref{e6}, thanks to Gronwall's lemma.
	
	\medskip
	\noindent\underline{\it In case \eqref{Lb} holds}, we integrate \eqref{e7} in time to obtain
	\begin{equation}\label{e8}
	\begin{aligned}
		&\sup_{t\in (0,T)}\L_p(t) + \frac{\alpha_p}{2}\sumk \intQT \bra{|\na u_k^{p/2}|^2 + |u_k|^p}dxdt\\
		&\leq \L_p(0) + C_pT + C_p\sumk\intQT u_k^{p-1+r}dxdt. 
	\end{aligned}
	\end{equation}
	Denote by $y_k:= u_k^{p/2}$. The left hand side of \eqref{e8} can be estimated below by
	\begin{equation}\label{e8_1}
		\text{LHS of (\ref{e8})} \geq C\sumk\bra{\|y_k\|_{L^{\infty}(0,T;\LO{2})}^2 + \|y_k\|_{L^2(0,T;H^1(\Omega))}^2}.
	\end{equation}
	For the right hand side of \eqref{e8}, we first consider 
	\begin{equation}\label{be_vt}
	\vt > p-1+r
	\end{equation}
	as a constant to be determined later. Of course we are only interested in the case when $p-1+r>b$, otherwise the right hand side of \eqref{e8} is bounded thanks to \eqref{Lb}. By the interpolation inequality we have
	\begin{equation}\label{e9}
		\intQT u_k^{p-1+r}dxdt = \|u_k\|_{\LQ{p-1+r}}^{p-1+r}\leq \|u_k\|_{\LQ{b}}^{\theta(p-1+r)}\|u_k\|_{\LQ{\vt}}^{(1-\theta)(p-1+r)},
	\end{equation}
	where $\theta\in (0,1)$ satisfies
	\begin{equation*}
		\frac{1}{p-1+r} = \frac{\theta}{b} + \frac{1-\theta}{\vt},
	\end{equation*}
	which implies 
	\begin{equation*}
		(1-\theta)(p-1+r) = \frac{\vt(p-1+r-b)}{\vt - b}.
	\end{equation*}
	Using this, and taking into account \eqref{Lb}, \eqref{e9} implies
	\begin{equation}\label{e9_1}
	\begin{aligned}
		\intQT u_k^{p-1+r}dxdt &\leq \F(T)^{\theta(p-1+r)}\|u_k\|_{\LQ{\vt}}^{\frac{\vt(p-1+r-b)}{\vt - b}}\\
		&= \F(T)^{\theta(p-1+r)}\bra{\intQT u_k^{\vt}dxdt}^{\frac{p-1+r-b}{\vt - b}}\\
		&= C_T\bra{\intQT y_k^{\frac{2\vt}{p}}dxdt}^{\frac{p-1+r-b}{\vt - b}}.
	\end{aligned}
	\end{equation}
	For $p$ large enough, we choose $\vt$ close enough to (but bigger than) $p-1+r$ so that $H^1(\Omega)\hookrightarrow \LO{\frac{2\vt}{p}}$. That means $\vt$ is arbitrary for $n\leq 2$ and
	\begin{equation}\label{up_vt1}
		\frac{2\vt}{p} \leq \frac{2n}{n-2} \Leftrightarrow \vt \leq \frac{pn}{n-2} \quad\text{ for }\quad n\geq 3.
	\end{equation}
	Thus, we can use the Gagliardo-Nirenberg's inequality to estimate
	\begin{equation}\label{e10}
		\intO y_k^{\frac{2\vt}{p}}dx = \|y_k\|_{\LO{\frac{2\vt}{p}}}^{\frac{2\vt}{p}} \leq C\|y_k\|_{H^1(\Omega)}^{\alpha\cdot \frac{2\vt}{p}}\|y_k\|_{\LO{2}}^{(1-\alpha)\cdot \frac{2\vt}{p}}
	\end{equation}
	where $\alpha\in (0,1)$ satisfies
	\begin{equation*}
		\frac{p}{2\vt} = \bra{\frac 12 - \frac 1n}\alpha + \frac{1-\alpha}{2}.
	\end{equation*}
	From this
	\begin{equation*}
		\alpha\cdot\frac{2\vt}{p} = \frac{n(\vt - p)}{p} \quad \text{ and } \quad (1-\alpha)\cdot\frac{2\vt}{p} = \frac{np - (n-2)\vt}{p}.
	\end{equation*}
	Therefore, we obtain from \eqref{e10} that
	\begin{equation*}
		\intO y_k^{\frac{2\vt}{p}}dx \leq C\|y_k\|_{H^1(\Omega)}^{\frac{n(\vt-p)}{p}}\|y_k\|_{\LO{2}}^{\frac{np-(n-2)\vt}{p}}.
	\end{equation*}
	It follows that
	\begin{align*}
		\intQT y_k^{\frac{2\vt}{p}}dxdt &\leq C\int_0^T\|y_k\|_{H^1(\Omega)}^{\frac{n(\vt-p)}{p}}\|y_k\|_{\LO{2}}^{\frac{np-(n-2)\vt}{p}} dt\\
		&\leq C\|y_k\|_{L^{\infty}(0,T;\LO{2})}^{\frac{np-(n-2)\vt}{p}}\int_0^T\|y_k\|_{H^1(\Omega)}^{\frac{n(\vt-p)}{p}}dt.
	\end{align*}
	We choose 
	\begin{equation}\label{up_vt2}
	\frac{n(\vt - p)}{p}\leq 2 \Leftrightarrow \vt \leq \frac{p(n+2)}{n},
	\end{equation}
	which is possible since $p - 1 + r < p + \frac{2p}{n}$ for $p$ large enough. Thus, by H\"older's inequality,
	\begin{equation*}
		\intQT y_k^{\frac{2\vt}{p}}dxdt \leq C_T\|y_k\|_{L^{\infty}(0,T;\LO{2})}^{\frac{np-(n-2)\vt}{p}}\|y_k\|_{L^2(0,T;H^1(\Omega))}^{\frac{n(\vt-p)}{p}}.
	\end{equation*}
	Inserting this into \eqref{e9_1} yields
	\begin{equation}\label{e11}
		\intQT u_k^{p-1+r}dxdt \leq C_T\bra{\|y_k\|_{L^{\infty}(0,T;\LO{2})}^{\frac{np-(n-2)\vt}{p}}\|y_k\|_{L^2(0,T;H^1(\Omega))}^{\frac{n(\vt-p)}{p}}}^{\frac{p-1+r-b}{\vt - b}}.
	\end{equation}
	We check that we can choose $\vt$ such that
	\begin{equation}\label{e12}
		\frac{p-1+r-b}{\vt - b}\cdot\bra{\frac{np-(n-2)\vt}{p} + \frac{n(\vt-p)}{p}} < 2.
	\end{equation}
	Indeed, this is equivalent to 
	\begin{align}\label{be_vt2}
		\frac{p-1+r-b}{\vt - b}\cdot \frac{2\vt}{p}<2 &\Leftrightarrow (p-1+r-b)\vt < (\vt-b)p\nonumber\\
		&\Leftrightarrow \vt > \frac{bp}{1+b-r}.
	\end{align}
	We now check that we can choose $\vt$ which satisfies all the conditions \eqref{be_vt}, \eqref{up_vt1}, \eqref{up_vt2} and \eqref{be_vt2}. This is fulfilled provided
	\begin{equation*}
		\frac{bp}{1+b-r}<\frac{p(n+2)}{n}, \quad\text{which is equivalent to}\quad r < 1 + \frac{2b}{n+2},
	\end{equation*}
	and this is exactly our assumption \eqref{assump_r}. Now by \eqref{e12}, we can use Young's inequality of the form
	\begin{equation*}
		x^{\lam_1}y^{\lam_2} \leq \eps(x^2+y^2) + C_{\eps} \quad \text{ for } \quad \lam_1+\lam_2 < 2,
	\end{equation*}
	to estimate \eqref{e11} further as
	\begin{equation*}
		\intQT u_k^{p-1+r}dxdt \leq \eps\bra{\|y_k\|_{L^{\infty}(0,T;\LO{2})}^2 + \|y_k\|_{L^2(0,T;H^1(\Omega))}^2} + C_{T,\eps}.
	\end{equation*}
	Thus
	\begin{equation*}
		\text{RHS of (\ref{e8})} \leq \L_p(0) + \eps\sumk \bra{\|y_k\|_{L^{\infty}(0,T;\LO{2})}^2 + \|y_k\|_{L^2(0,T;H^1(\Omega))}}^2 + C_{T,\eps}.
	\end{equation*}
	Combining this with \eqref{e8_1} gives us the desired estimate \eqref{e6}.
	
	\medskip
	For the uniform-in-time bounds, we use similar arguments to the proof of Theorem \ref{thm1} study \eqref{Sys} (multiplied by a time-truncated function $\psi_\tau$) on each interval $(\tau,\tau+2)$ to ultimately obtain
	\begin{equation*}
		\sup_{\tau\in \mathbb N}\|u_i\|_{L^{\infty}(\Omega\times(\tau,\tau+2))} < +\infty \quad \forall i=1,\ldots, m.
	\end{equation*}
	We leave the details for the interested reader.
\end{proof}
\subsection{Other boundary conditions or quasilinear systems}\label{subs3}
First, we state the definition of weak solutions to \eqref{Sys_other_bc} as it is different from that of \eqref{Sys}.
\begin{definition}\label{def_diff_bc}
	A vector of non-negative concentrations $u = (u_1, \ldots, u_m)$ is called a weak solution to \eqref{Sys_other_bc} on $(0,T)$ if
	\begin{equation*}
		u_i\in C([0,T]; \LO{2})\cap L^2(0,T;H^1(\Omega)), \quad F_i(u)\in L^2(0,T;\LO{2}),
	\end{equation*}
	with $u_i(\cdot,0)= u_{i,0}(\cdot)$ for all $i=1,\ldots, m$, and for any test function $\varphi \in L^2(0,T;H^1(\Omega))$ with $\pa_t\varphi \in L^2(0,T;H^{-1}(\Omega))$, it holds that
	\begin{equation}\label{weak_bc}
	\begin{aligned}
		\intO u_i(x,t)\varphi(x,t)dx\bigg|_{t=0}^{t=T} - \intQT u_i\pa_t\varphi dxdt + \intQT D_i(x,t)\na u_i \cdot \na\varphi dxdt\\
		+\alpha_i\int_0^T\int_{\pa\Omega}u_i\varphi d\H^{n-1}dt = \intQT F_i(x,t,u)\varphi dxdt.
	\end{aligned}
	\end{equation}
\end{definition}
\begin{proof}[Proof of Theorem \ref{thm3_1}]
	The proof of this theorem is similar to that of Theorems \ref{thm1} and \ref{thm3}, except for the fact that the $L^1$-norm can be obtained in a different, and easier, way. Since $\varphi \equiv 1$ is an admissible test function we get from \eqref{weak_bc} and \eqref{F3} that
	\begin{align*}
		\sumi \intO c_iu_i(\cdot,t) dx\bigg|_{t=0}^{t=T} + \sumi c_i\alpha_i\int_0^T\int_{\pa\Omega} u_id\H^{n-1}dt\\
		\leq K_1\intQT u_i dxdt + K_2|\Omega|T.
	\end{align*}
	The $L^1$-bound, which is uniform in time in case $K_1<0$ or $K_1 = K_2 =0$, then follows immediately from the Gronwall's inequality.
\end{proof}
\begin{proof}[Proof of Theorem \ref{thm3_2}]
	The definition of weak solutions of \eqref{Sys_quasilinear} is similar to Definition \ref{def_weak}. Note that the Galerkin method in the proof of Lemma \ref{lem:existence_truncated} is not applicable due to the nonlinear dependence of $A_i(x,t,u)$ on $u$. We resort to a different strategy.
	
	For fixed $\eps>0$, we consider the truncated system: for all $i=1,\ldots, m$,
	\begin{equation}\label{Sys_truncated_quasilinear}
		\begin{cases}
			\pa_t \ue_i - \na\cdot\bra{A_i^\eps(x,t,\ue)\na \ue_i} = F_i^\eps(x,t,\ue), & x\in\Omega, t>0,\\
			u_i(x,t) = 0, & x\in\pa\Omega, t>0,\\
			u_i(x,0) = \ue_{i,0}(x), & x\in\Omega,
		\end{cases}
	\end{equation}
	where for all $i=1,\ldots, m$,
	\begin{align*}
		\bullet\quad & F_i^{\eps}(x,t,\ue) = F_i(x,t,\ue)\sbra{ 1+\eps\sumj |F_j(x,t,\ue)|}^{-1},\\
		\bullet\quad& \ue_{i,0}\in C^2(\Omega)\cap C(\bar \Omega) \; \text{ and } \; \lim_{\eps\to 0}\|\ue_{i,0} - u_{i,0}\|_{\LO{\infty}} = 0,\\
		\bullet\quad& \Omega\times \R_+\times \R^m \ni (x,t,\omega)\mapsto A_i^\eps(x,t,\omega) \text{ is H\"older continuous in each component},\\
		&\quad \frac{\wh\lam}{2}|\xi|^2 \leq \xi^\top A_i^\eps(x,t,\omega)\xi \quad \forall (x,t,\omega)\in \Omega\times\R_+\times \R_+^m, \; \forall \xi\in\R^n,\\
		&\quad A_i^\eps(x,t,\omega)\to A_i(x,t,\omega) \quad \forall \omega\in\R^m, \; \text{ a.e.} (x,t)\in\Omega\times\R_+.
	\end{align*}
	Thanks to these regularization, the existence of a weak solution $\ue = (\ue_i)_{i=1,\ldots,m}$ to \eqref{Sys_truncated_quasilinear}  follows from classical results, see e.g. \cite{palagachev2011quasilinear}. Moreover, this solution is unique \cite{nittka2013quasilinear}. Thanks to this uniqueness, we can also get the non-negativity of $\ue$ using similar computations to that of Lemma \ref{lem:existence_truncated}.  Following the arguments of Lemma \ref{L1-bound}, Lemma \ref{Lp-bound}, Proposition \ref{pro1}, and Theorem \ref{thm3} we obtain that
	\begin{equation*}
		\|\ue_i\|_{\LQ{\infty}} \leq C_T \quad \forall i=1,\ldots, m.
	\end{equation*}	
	Note that this is possible since we used only the uniform ellipticity of $A_i^\eps$ to derive this bound, which is uniformly in $\eps>0$. From \eqref{Sys}, we obtain $\{\ue_i\}_{\eps>0}$ is bounded in $L^2(0,T;H_0^1(\Omega))$ and $\{\pa_t\ue_i\}_{\eps>0}$ is bounded in $L^2(0,T;H^{-1}(\Omega))$. The Aubin-Lions lemma gives strong in $L^2(Q_T)$, and consequently point-wise, convergence $\ue_i \to u_i$ (up to a subsequent) for all $i=1,\ldots, m$. Passing to the limit $\eps\to0$ for the weak formulation of \eqref{Sys_truncated_quasilinear}, we only have to show that
	\begin{equation*}
		\intQT A_i^\eps(x,t,\ue)\na \ue_i\cdot \na \varphi dxdt \xrightarrow{\eps \to 0} \intQT A_i(x,t,u)\na u_i \cdot \na \varphi dxdt.
	\end{equation*}
	This is guaranteed thanks to the boundedness of $A_i^\eps$, $A_i^\eps(x,t,\ue)\to A_i(x,t,u)$ almost everywhere, and $\ue_i \rightharpoonup u_i$ weakly in $L^2(0,T;H_0^1(\Omega))$. The uniform-in-time bound in sup-norm can be obtained in similar ways as in previous theorems, so we omit the details.
	
	If $A_i$ is merely bounded in $(x,t)$, it seems not possible to show uniqueness of weak solutions. Under the stronger assumption $(iv)$ that, $A(\cdot,\cdot,\omega)$ is H\"older continuous on $\Omega\times \R_+$, one can show that the gradient $\na u_i$ is sup-norm bounded, and the uniqueness follows. We refer the interested reader to \cite{nittka2013quasilinear} for more details.
	\end{proof}
\section{Applications to a model of an infectious disease}\label{sec:applications}
{\color{black}
We apply our theoretical results to a model for the spread of infectious disease within a host population that occupies a highly heterogeneous habitat. Related problems have been considered in many works in the literature, see e.g. \cite{fitzgibbon2001mathematical,bendahmane2002existence,eisenberg2013cholera,yamazaki2017global,yin2020reaction} and the references therein. Here we extend the model considered in \cite{yin2020reaction}. We consider a population confined in a bounded domain $\Omega\subset\mathbb R^n$, $n\geq 1$. The population is compartmentalized into the susceptible class $S$, the infected class $I$, and the recovered class $R$, whose densities are denoted by $s(x,t)$, $i(x,t)$ and $r(x,t)$, respectively. Moreover, we denote by $b(x,t)$ the density of the pathogen associated with the disease. We consider the following system
\begin{equation}\label{irs_Sys}
\begin{cases}
	\pa_t s = \na\cdot(D_S(x,t)\na s) - \sigma_I(x)si - \sigma_B(x) sb + \gamma r, &x\in\Omega, \; t>0,\\
	\pa_t i = \na\cdot(D_I(x,t)\na i) + \sigma_I(x) si + \sigma_B(x) sb - (\lam+\alpha)i, &x\in\Omega, \; t>0,\\
	\pa_t r = \na\cdot(D_R(x,t)\na r) + \lam i - \gamma r, &x\in\Omega, \; t>0,\\
	\pa_t b = \na\cdot(D_B(x,t)\na b + c(x,t)b) + \phi(x) i - \delta b, &x\in\Omega, \; t>0.
\end{cases}
\end{equation}
where we assume
\begin{itemize}
	\item The diffusivities $D_S(x,t), D_I(x,t), D_R(x,t)$ and $D_B(x,t)$ are uniformly bounded on $\Omega$ and there exists a $d_*>0$ so that
	\begin{equation}\label{D}
		0 < d_* \leq \min_{x\in\Omega, t>0} \{D_S(x,t), D_I(x,t), D_R(x,t), D_B(x,t)\}.
	\end{equation}
	\item The exists a $c^* > 0$ so that 
	\begin{equation}\label{C}
		\max_{1\leq j\leq n}\sup_{t>0}\|c_j(t)\|_{\LO{\infty}} < c^*.
	\end{equation}
	\item There exist $\sigma_*$ and $\sigma^*$ such that 
	\begin{equation}\label{sigma}
		0<\sigma_* \leq \min_{x\in\Omega} \{\sigma_I(x), \sigma_B(x)\} \leq \max_{x\in\Omega}\{\sigma_I(x),\sigma_B(x)\} \leq \sigma^*.
	\end{equation}
\end{itemize}
The system \eqref{irs_Sys} is subject to no-flux boundary condition 
for $x\in\pa\Omega$, $t>0$,
\begin{equation*}
	D_S\na s\cdot {\nu} = D_I\na i\cdot \nu = D_R\na r\cdot\nu = (D_B\na b + c(x,t)b)\cdot \nu = 0,
\end{equation*}
and bounded, non-negative initial data for $x\in\Omega$,
\begin{equation*}
	s(x,0) = s_0(x), \; i(x,0) = i_0(x), \; r(x,0) = r_0(x), \; b(x,0) = b_0(x).
\end{equation*}
For more discussion about the modeling of \eqref{irs_Sys}, we refer the reader to \cite{yin2020reaction}.}

We are in position to show the global existence, boundedness and asymptotic behavior of solutions to \eqref{irs_Sys}. It is emphasized that the uniform-in-time boundedness plays an important role in establishing the large time behavior of the solution.
\begin{theorem}\label{thm4}
	Assume \eqref{D}, \eqref{C}, \eqref{sigma}, positive parameters $\gamma, \lam, \alpha, \delta>0$, $\phi(x)\geq 0$ a.e. $x\in\Omega$ and
	\begin{equation}\label{cond_phi}
		\|\phi\|_{\LO{\infty}} \le \alpha.
	\end{equation}
	For any bounded, non-negative initial data, there exists a unique global bounded weak solution to \eqref{Sys} which is bounded uniformly in time, i.e.
	\begin{equation}\label{irs_bound}
		\sup_{t\geq 0}\big\{\|s(t)\|_{\LO{\infty}}, \|i(t)\|_{\LO{\infty}}, \|r(t)\|_{\LO{\infty}}, \|b(t)\|_{\LO{\infty}} \} < +\infty.
	\end{equation}
	Moreover, we have the following asymptotic behavior for any $1<p<\infty$,
	\begin{equation}\label{decay}
		\lim_{t\to\infty}\bra{\|i(t)\|_{\LO{p}} + \|r(t)\|_{\LO{p}} + \|b(t)\|_{\LO{p}} + \|s(t)-s_\infty\|_{\LO{p}}} = 0
	\end{equation}
	where 
	\begin{equation*}
		0< s_\infty = \intO(s_0+i_0+r_0)dx - \alpha\int_0^\infty \intO i(x,z)dxdz < +\infty.
	\end{equation*}
\end{theorem}
One can interpret the convergence of $s$ as meaning that the final susceptible population equals to the difference of the initial total population $\intO(s_0 + i_0 + r_0)dx$ and the total  mortality over time $\alpha\int_0^{\infty}\intO i(x,z)dxdz$, which is expected from realistic situations.
\begin{proof}
	Denote by $u=(s,i,r,b)$ the vector of unknowns, and $F_s, F_i, F_r$ and $F_b$ the nonlinearities in the equations of $s, i, r$ and $b$, respectively. Thanks to \eqref{cond_phi}, we have $F_s(u) \leq \gamma r$, $F_s(u) + F_i(u)\leq \gamma r$, $F_s(u) + F_i(u) + F_r(u) \leq 0$, and $F_s(u) + F_i(u) + F_r(u) + F_b(u) \leq 0$. The global existence and uniform boundedness \eqref{irs_bound} of a unique weak solution to \eqref{irs_Sys} then follow from Theorem \ref{thm3_1} with slight modifications (due to the different boundary condition of $b$). 
	
	\medskip
	We now turn to the large time behavior. By summing the equations of $s, i, r$ and taking $\varphi \equiv 1$ as a test function, we have
	\begin{equation}\label{ee0}
		\intO (s(x,t) + i(x,t) + r(x,t)) + \alpha \int_0^t\intO i(x,z)dxdz = \intO (s_0(x) + i_0(x) + r_0(x))dx. 
	\end{equation}
	Due to the non-negativity of $i$, it follows that
		$\int_0^\infty \|i(z)\|_{\LO{1}}dz < +\infty$,
	which implies
	\begin{equation*}
		\lim_{t\to\infty}\int_t^{t+1}\|i(z)\|_{\LO{1}}dz = 0.
	\end{equation*}
	By integrating the equations of $r$ and $b$ on $\Omega\times(0,t)$, we obtain easily
	\begin{equation*}
		\lim_{t\to\infty}\int_t^{t+1}\|r(z)\|_{\LO{1}}dz = 0\quad \text{ and }\quad \lim_{t\to\infty}\int_t^{t+1}\|b(z)\|_{\LO{1}}dz = 0.
	\end{equation*}
	Thanks to the bound in $\LO{\infty}$-norm \eqref{irs_bound} and the inequality $\|f\|_{\LO{p}}^p \leq \|f\|_{\LO{\infty}}^{p-1}\|f\|_{\LO{1}}$ we get for any $p>1$
	\begin{equation}\label{integration_Lp_decay}
		\lim_{t\to\infty}\int_t^{t+1}\bra{\|i(z)\|_{\LO{p}}^p + \|r(z)\|_{\LO{p}}^p + \|b(z)\|_{\LO{p}}^p}dz = 0.
	\end{equation}
	We now show the decay of $b$ in $\LO{p}$-norm for $p>1$. For $\tau\in \mathbb N$, recall the time cut-off function $\varphi_\tau\in C^\infty(\R)$ satisfying $\varphi_\tau(\cdot) = 0$ on $(-\infty,\tau]$ an $\varphi_\tau(\cdot) = 1$ on $[\tau+1,\infty)$. It's clear that $\pt u = (\pt s,\pt i, \pt r, \pt b)$ is a also a weak solution to \eqref{irs_Sys}, which satisfies $\pt y(x,\tau) = 0$ for $y\in \{s,i,r,b\}$. In particular, by taking $p\pt^p b^{p-1}$ as a test function for the equation of $b$ we obtain
	\begin{equation}\label{ee}
	\begin{aligned}
		\|(\pt b)(\tau+1)\|_{\LO{p}}^p + p(p-1)\int_{\tau}^{\tau+1}\intO\pt^pb^{p-2}(D_B(x,z)\na b(x,z))\cdot \na b(x,z) dxdz \\
		- \delta p\int_{\tau}^{\tau+1}\intO \pt^pb(x,z)^pdxdz= -p\int_{\tau}^{\tau+1}\intO c(x,z) \pt^p b^{p-1}\na b dxdz\\
		+p\int_{\tau}^{\tau+1}\intO \pt^p\pt' b^{p}dxdz + p\int_\tau^{\tau+1}\intO\pt^p \phi i b^{p-1}dxdz.
	\end{aligned}
	\end{equation}
	By H\"older's and Young's inequalities, we have
	\begin{align*}
		\left|p\int_{\tau}^{\tau+1}\!\!\intO c\pt^p b^{p-1}\na b dxdz \right|
		\leq p(p-1)\frac{d_*}{2}\int_{\tau}^{\tau+1}\!\!\intO \pt^pb^{p-2}|\na b|^2dxdz + C\int_{\tau}^{\tau+1}\!\!\intO |b|^pdxdz.
	\end{align*}
	Inserting this into \eqref{ee}, using \eqref{D} and H\"older's inequality yield
	\begin{equation*}
		\|b(\tau+1)\|_{\LO{p}}^p \leq C\int_\tau^{\tau+1}\bra{\|i(z)\|_{\LO{p}}^p + \|b(z)\|_{\LO{p}}^p}dz
	\end{equation*}
	where $C$ is independent of $\tau$. Thus, it follows from \eqref{integration_Lp_decay} that $\lim_{t\to\infty}\|b(t)\|_{\LO{p}} = 0$. By using similar arguments, $\lim_{t\to\infty}\|r(t)\|_{\LO{p}} = 0$. 
	For $i$, we have 
	\begin{align*}
		\|i(\tau+1)\|_{\LO{p}}^p &\leq C\int_{\tau}^{\tau+1}\intO\bra{|i(z)|^p + |s(z)||i(z)|^p + |s(z)||b(z)||i(z)|^{p-1}}dxdz\\
		&\leq C\bra{1+\sup_{t\geq 0}\|s(t)\|_{\LO{\infty}}}\int_{\tau}^{\tau+1}\bra{\|i(z)\|_{\LO{p}}^p + \|b(z)\|_{\LO{p}}^p}dz
	\end{align*}
	which implies $\lim_{t\to\infty}\|i(t)\|_{\LO{\infty}} = 0$. It remains to show the asymptotic of $s$. Denote by $\bar{s}(t) = \frac{1}{|\Omega|}\intO s(x,t)dx$. Direct calculations give
	\begin{align*}
		\frac 12\frac{d}{dt}\|s - \bar s\|_{\LO{2}}^2 \leq -d_*\intO |\na s|^2dx + \|s-\bar s\|_{\LO{2}}\|\sigma_I si + \sigma_B sb - \gamma \bar{s}_t\|_{\LO{2}}.
	\end{align*}
	By using Young's inequality, it yields
	\begin{align*}
		&\|s - \bar s\|_{\LO{2}}\|\sigma_I si + \sigma_B sb - \gamma\bar s_t\|_{\LO{2}}\\
		&\leq \frac{d_*C_{\text{PW}}}{2}\|s-\bar s\|_{\LO{2}}^2 + \frac{1}{2d_*C_{\text{PW}}}\|\sigma_I si + \sigma_B sb - \gamma\bar s_t\|_{\LO{2}}^2.
	\end{align*}
	Thus, by the Poincar\'e-Wirtinger inequality $\intO |\na s|^2dx \geq C_{\text{PW}}\|s - \bar s\|_{\LO{2}}^2$, we have
	\begin{equation}\label{ee3}
		\frac{d}{dt}\|s - \bar s\|_{\LO{2}}^2 + d_*C_{\text{PW}}\|s - \bar s\|_{\LO{2}}^2 \leq C\|\sigma_I si + \sigma_B sb - \gamma \bar{s}_t\|_{\LO{2}}^2.
	\end{equation}
	From the uniform bounds of $\sigma_I, \sigma_B, s$, and the decay of $i$ and $b$, we have
	\begin{equation}\label{ee4}
		\lim_{t\to\infty}\|\sigma_I si + \sigma_B sb\|_{\LO{2}} \leq \lim_{t\to\infty}\sigma^*\|s(t)\|_{\LO{\infty}}\bra{\|i(t)\|_{\LO{2}} + \|b(t)\|_{\LO{2}}} = 0.
	\end{equation}
	From the equation of $s$,
	\begin{align*}	
		\left|\bar{s}_t\right| \leq \bra{\sigma^* + \gamma}\bra{\|s\|_{\LO{\infty}}(\|i\|_{\LO{1}} + \|b\|_{\LO{1}}) + \gamma\|r\|_{\LO{1}}},
	\end{align*}
	and thus
	\begin{equation}\label{ee5}
		\|\bar s_t(z)\|_{\LO{2}}^2 = |\Omega||\bar s_t(z)|^2 \xrightarrow{z\to\infty} 0. 
	\end{equation}
	From \eqref{ee4} and \eqref{ee5}, \eqref{ee3} implies that
	\begin{equation}\label{ee6}
		\lim_{t\to\infty}\|s(t) - \bar{s}(t)\|_{\LO{2}}^2 = 0.
	\end{equation}
	Finally, by letting $t\to\infty$ in \eqref{ee0}, it follows
	\begin{equation*}
		\lim_{t\to\infty}\bar{s}(t) = \lim_{t\to\infty}\intO s(x,t)dx = \intO(s_0 + i_0 + r_0)dx - \alpha\int_0^\infty\intO i(x,z)dxdz= s_\infty,
	\end{equation*}
	which, in combination with \eqref{ee6} implies
	\begin{equation*}
		\lim_{t\to\infty}\|s(t)-s_\infty\|_{\LO{2}} = 0.
	\end{equation*}
	The uniform bound of $s$ in $\LO{\infty}$ and interpolation inequality imply the behavior \eqref{decay}.
\end{proof}

\section{Two technical lemmas}\label{sec:technical}
\begin{lemma}\label{Hp-lem7}
	Suppose $m_1\in \mathbb N$, $\theta= (\theta_1,\ldots, \theta_{m})$, where $\theta_1,...,\theta_{m}$ are positive real numbers, $\beta\in \mathbb Z_+^{m}$, and $\H_p[u]$ is defined in (\ref{Hp}). Then
	$$\frac{\partial}{\partial t}\H_0[u](t)=0,\text{\quad  }\frac{\partial}{\partial t}\H_1[u](t)=\sum_{j=1}^{m}\theta_j\frac{\partial}{\partial t}u_j(t),$$  
	and for $p\in\mathbb N$ such that $p\ge 2$, 
	\begin{equation*}
	\frac{\partial}{\partial t}\H_p[u](t) = \sum_{|\beta| = p-1}\begin{pmatrix} p\\ \beta \end{pmatrix} \theta^{\beta^2}u(t)^{\beta}\sum_{j=1}^{m}\theta_j^{2\beta_j+1}\frac{\partial}{\partial t}u_j(t).
	\end{equation*}
\end{lemma}

\begin{proof}
	The results for $\H_0[u](t)$ and $\H_1[u](t)$ are trivial. The same is true for the case when $m=1$. Suppose $p\ge 2$ and $m_1\ge 2$. We proceed by induction on the value $m$, and assume $k\in \mathbb N$ such that the result is true for $m=k$. Suppose $m=k+1$ and denote
	$$\tilde\beta=(\beta_2,...,\beta_{m})\text{ and }\tilde u=(u_2,...,u_{m}).$$
	Then we can rewrite $\H_p[u]$ as
	\begin{equation}\label{Hp-eq1}
	\H_p[u] = \sum_{\beta_1=0}^p\frac{1}{\beta_1!}\theta_1^{\beta_1^2}u_1^{\beta_1}\sum_{|\tilde\beta|=p-\beta_1}\begin{pmatrix}
	p\\ \tilde\beta\end{pmatrix}\tilde\theta^{\tilde\beta^2}\tilde u^{\tilde\beta}.
	\end{equation}
	Consequently,
	\begin{align}\label{Hp-eq2}
	\frac{\partial}{\partial t}\H_p[u] &= \sum_{\beta_1=1}^p\frac{1}{\beta_1!}\theta_1^{\beta_1^2}\beta_1u_1^{\beta_1-1}\frac{\partial}{\partial t}u_1\sum_{|\tilde\beta|=p-\beta_1}\begin{pmatrix}
	p\\ \tilde\beta\end{pmatrix}\tilde\theta^{\tilde\beta^2}\tilde u^{\tilde\beta}\nonumber\\
	&+\sum_{\beta_1=0}^p\frac{1}{\beta_1!}\theta_1^{\beta_1^2}u_1^{\beta_1}\frac{\partial}{\partial t}\left(\sum_{|\tilde\beta|=p-\beta_1}\begin{pmatrix}
	p\\ \tilde\beta\end{pmatrix}\tilde\theta^{\tilde\beta^2}\tilde u^{\tilde\beta}\right)\nonumber\\
	&=\sum_{\beta_1=1}^p\frac{1}{\beta_1!}\theta_1^{\beta_1^2}\beta_1u_1^{\beta_1-1}\frac{\partial}{\partial t}u_1\sum_{|\tilde\beta|=p-\beta_1}\begin{pmatrix}
	p\\ \tilde\beta\end{pmatrix}\tilde\theta^{\tilde\beta^2}\tilde u^{\tilde\beta}\nonumber\\
	&+\sum_{\beta_1=0}^{p-1}\frac{1}{\beta_1!}\theta_1^{\beta_1^2}u_1^{\beta_1}\frac{p!}{(p-\beta_1)!}\frac{\partial}{\partial t}H_{p-\beta_1}[\tilde u].
	\end{align}
	Now, from our induction hypothesis, 
	\begin{equation}\label{Hp-eq3}
	\frac{\partial}{\partial t}\H_{p-\beta_1}[\tilde u] = \sum_{|\tilde\beta| = p-\beta_1-1}\begin{pmatrix} p-\beta_1\\ \tilde\beta \end{pmatrix} \tilde\theta^{\tilde\beta^2}\tilde u^{\tilde\beta}\sum_{j=1}^{m-1}\tilde\theta_j^{2\tilde\beta_j+1}\frac{\partial}{\partial t}\tilde u_j.
	\end{equation}
	Therefore, substituting (\ref{Hp-eq3}) into (\ref{Hp-eq2}), and noting that $\tilde u_j=u_{j+1}$ and $\tilde\theta_j=\theta_{j+1}$, gives 
	\begin{align}\label{Hp-eq4}
	\frac{\partial}{\partial t}\H_p[u] &=\sum_{\beta_1=1}^p\frac{1}{\beta_1!}\theta_1^{\beta_1^2}\beta_1u_1^{\beta_1-1}\frac{\partial}{\partial t}u_1\sum_{|\tilde\beta|=p-\beta_1}\begin{pmatrix}
	p\\ \tilde\beta\end{pmatrix}\tilde\theta^{\tilde\beta^2}\tilde u^{\tilde\beta}\nonumber\\
	&+\sum_{\beta_1=0}^{p-1}\frac{1}{\beta_1!}\theta_1^{\beta_1^2}u_1^{\beta_1}\frac{p!}{(p-\beta_1)!}\sum_{|\tilde\beta| = p-\beta_1-1}\begin{pmatrix} p-\beta_1\\ \tilde\beta \end{pmatrix} \tilde\theta^{\tilde\beta^2}\tilde u^{\tilde\beta}\sum_{j=1}^{m-1}\tilde\theta_j^{2\tilde\beta_j+1}\frac{\partial}{\partial t}\tilde u_j\nonumber\\
	&=\sum_{\beta_1=0}^{p-1}\frac{1}{\beta_1!}\theta_1^{(\beta_1+1)^2}u_1^{\beta_1}\frac{\partial}{\partial t}u_1\sum_{|\tilde\beta|=p-\beta_1-1}\begin{pmatrix}
	p\\ \tilde\beta\end{pmatrix}\tilde\theta^{\tilde\beta^2}\tilde u^{\tilde\beta}\nonumber\\
	&+\sum_{\beta_1=0}^{p-1}\frac{1}{\beta_1!}\theta_1^{\beta_1^2}u_1^{\beta_1}\frac{p!}{(p-\beta_1)!}\sum_{|\tilde\beta| = p-\beta_1-1}\begin{pmatrix} p-\beta_1\\ \tilde\beta \end{pmatrix} \tilde\theta^{\tilde\beta^2}\tilde u^{\tilde\beta}\sum_{j=1}^{m-1}\tilde\theta_j^{2\tilde\beta_j+1}\frac{\partial}{\partial t}\tilde u_j\nonumber\\
	&=\sum_{|\beta|=p-1}\begin{pmatrix}p\\\beta\end{pmatrix}\theta^{\beta^2}u^{\beta}\theta_1^{2\beta_1+1}\frac{\partial}{\partial t}u_1
	+\sum_{|\beta|=p-1}\begin{pmatrix}p\\\beta\end{pmatrix}\theta^{\beta^2}u^{\beta}\sum_{j=2}^{m}\theta_j^{2\beta_j+1}\frac{\partial}{\partial t}u_j\nonumber\\
	&=\sum_{|\beta|=p-1}\begin{pmatrix}p\\\beta\end{pmatrix}\theta^{\beta^2}u^{\beta}\sum_{j=1}^{m}\theta_j^{2\beta_j+1}\frac{\partial}{\partial t}u_j.
	\end{align}\end{proof}

We also need a second technical result. We have recently proved this in the far simpler case when $(a_{i,j}^k)=d_kI_{n\times n}$ in \cite{morgan2021global}. The general case is handled below.

\begin{lemma}\label{Hp-lem8}
	Suppose $m\in \mathbb N$, $\theta= (\theta_1,\ldots, \theta_{m})$, where $\theta_1,...,\theta_{m}$ are positive real numbers, and let $\H_p[u]$ be defined in (\ref{Hp}). If $p\in\mathbb N$ such that $p\ge 2$, then
	$$\sum_{|\beta|=p-1}\begin{pmatrix}p\\ \beta\end{pmatrix}\theta^{\beta^2}\sum_{k=1}^{m}\theta_k^{2\beta_k+1}\left(A_k\nabla u_k\right)\cdot\nabla u^\beta =\sum_{|\beta|=p-2}\left(\begin{array}{c}
	p\\
	\beta
	\end{array}\right)\theta^{\beta^{2}}u^{\beta}\sum_{k=1}^{m}\sum_{r=1}^{m}C_{k,r}(\beta)\left(A_k\nabla u_k\right)\cdot\nabla u_r,$$
	where 
	\[
	C_{k,r}(\beta)=\begin{cases}
	\begin{array}{cc}
	\theta_{k}^{2\beta_{k}+1}\theta_{r}^{2\beta_{r}+1}, & k\ne r,\\
	\theta_{k}^{4\beta_{k}+4}, & k=r.
	\end{array}\end{cases}
	\]
\end{lemma}

\begin{proof}The result is easily verified when $m_1=1$, regardless of the choice of $p$, and for $p=2$, regardless of the choice of $m_1$. Suppose $p\ge 2$ and $m_1\ge 2$. We proceed by induction on the value $m_1$, and assume $k\in \mathbb N$ such that the result is is true for $m_1=k$. Suppose $m_1=k+1$ and (as in the proof of Lemma \ref{Hp-lem7}) denote
	$$\tilde\beta=(\beta_2,...,\beta_{m_1})\text{ and }\tilde u=(u_2,...,u_{m_1}).$$
	Also, whenever $w\in\mathbb{R}^n$ and $k\in\{1,...,n\}$, we denote entry $k$ in $w$ by $w_k$.
	Then
	\begin{align}\label{Hp-eq5}
	&\sum_{|\beta|=p-1}\begin{pmatrix}p\\ \beta\end{pmatrix}\theta^{\beta^2}\sum_{i=1}^{m_1}\theta_i^{2\beta_i+1}\left(A_i\nabla u_i\right)\cdot\nabla u^\beta \nonumber\\
	&=\sum_{\beta_1=0}^{p-1}\sum_{|\tilde\beta|=p-\beta_1-1}\begin{pmatrix}p\\ \beta\end{pmatrix}\theta_1^{\beta_1^2}\tilde\theta^{\tilde\beta^2}\left[\theta_1^{2\beta_1+1}\left(A_1\nabla u_1\right)\cdot\nabla u^\beta+\sum_{i=1}^{m_1-1}\tilde\theta_i^{2\tilde\beta_i+1}\left(A_{i+1}\nabla\tilde u_i\right)\cdot\nabla\left(u_1^{\beta_1}\tilde u^{\tilde\beta}\right)\right].
	\end{align}
	Note that for $1\le \beta_1\le p-2$ and $|\tilde\beta|=p-\beta_1-1$
	\begin{align}\label{Hp-eq6}
	\nabla\left(u_1^{\beta_1}\tilde u^{\tilde\beta}\right)=\beta_1 u_1^{\beta_1-1}\tilde u^{\tilde\beta}\nabla u_1+\sum_{j=1,\tilde\beta_j\ne 0}^{m_1-1}\tilde\beta_j u_1^{\beta_1} \tilde u^{\tilde\beta-e_j}\nabla \tilde u_j.
	\end{align}
	Therefore, from (\ref{Hp-eq5}) and (\ref{Hp-eq6}), we have
	\begin{align}\label{Hp-eq7}
	\sum_{|\beta|=p-1}\begin{pmatrix}p\\ \beta\end{pmatrix}\theta^{\beta^2}\sum_{i=1}^{m_1}\theta_i^{2\beta_i+1}\left(A_i\nabla u_i\right)\cdot\nabla u^\beta=\text{I}+\text{II},
	\end{align}
	where
	\begin{align}\label{Hp-eq8}
	\text{I}=\sum_{\beta_1=0}^{p-1}\sum_{|\tilde\beta|=p-\beta_1-1}\begin{pmatrix}p\\\beta\end{pmatrix}\theta_1^{\beta_1^2}\tilde\theta^{\tilde\beta^2} &\bigg[\theta_1^{2\beta_1+1}\left(A_1\nabla u_1\right)\cdot \nabla u^{\beta}\nonumber\\
	&\quad+\sum_{i=1,\beta_1\ne 0}^{m_1-1}\tilde\theta_i^{2\tilde\beta_i+1}\left(A_{i+1}\nabla \tilde u_i\right)\cdot\beta_1 u_1^{\beta_1-1}\tilde u^{\tilde\beta}\nabla u_1\bigg]
	\end{align}
	and
	\begin{align}\label{Hp-eq9}
	\text{II}=\sum_{\beta_1=0}^{p-2}\sum_{|\tilde\beta|=p-\beta_1-1}\begin{pmatrix}p\\\beta\end{pmatrix}\theta_1^{\beta_1^2}\tilde\theta^{\tilde\beta^2}\sum_{i,j=1,\tilde\beta_j\ne 0}^{m_1-1}\tilde\theta_i^{2\tilde\beta_i+1}\left(A_{i+1}\nabla \tilde u_i\right)\cdot\tilde\beta_ju_1^{\beta_1}\tilde u^{\tilde\beta-e_j}\nabla \tilde u_j.
	\end{align}
	Above, $e_j$ denotes row $j$ of the $(m_1-1)\times (m_1-1)$ identity matrix. We start with the analysis of II. We can rewrite
	\begin{align}\label{Hp-eq10}
	\text{II}&=\sum_{\beta_1=0}^{p-2}\frac{1}{\beta_1!}\theta_1^{\beta_1^2}u_1^{\beta_1}\sum_{|\tilde\beta|=p-\beta_1-1}\begin{pmatrix}p\\\tilde\beta\end{pmatrix}\tilde\theta^{\tilde\beta^2}\sum_{i,j=1,\tilde\beta_j\ne 0}^{m_1-1}\tilde\theta_i^{2\tilde\beta_i+1}\left(A_{i+1}\nabla \tilde u_i\right)\cdot\tilde\beta_j\tilde u^{\tilde\beta-e_j}\nabla \tilde u_j\nonumber\\
	&=\sum_{\beta_1=0}^{p-2}\frac{1}{\beta_1!}\theta_1^{\beta_1^2}u_1^{\beta_1}\sum_{|\tilde\beta|=p-\beta_1-1}\begin{pmatrix}p\\\tilde\beta\end{pmatrix}\tilde\theta^{\tilde\beta^2}\sum_{i,j=1,\tilde\beta_j\ne 0}^{m_1-1}\tilde\theta_i^{2\tilde\beta_i+1}\left(A_{i+1}\nabla \tilde u_i\right)\cdot \nabla\tilde u^{\tilde\beta}\nonumber\\
	&=\sum_{\beta_1=0}^{p-2}\frac{1}{\beta_1!}\theta_1^{\beta_1^2}u_1^{\beta_1}\frac{p!}{(p-\beta_1)!}\sum_{|\tilde\beta|=p-\beta_1-1}\begin{pmatrix}p-\beta_1\\\tilde\beta\end{pmatrix}\tilde\theta^{\tilde\beta^2}\sum_{i,j=1,\tilde\beta_j\ne 0}^{m_1-1}\tilde\theta_i^{2\tilde\beta_i+1}\left(A_{i+1}\nabla \tilde u_i\right)\cdot \nabla\tilde u^{\tilde\beta}\nonumber\\
	&=\sum_{\beta_1=0}^{p-2}\frac{1}{\beta_1!}\theta_1^{\beta_1^2}u_1^{\beta_1}\frac{p!}{(p-\beta_1)!}\sum_{|\tilde\beta|=p-\beta_1-2}\begin{pmatrix}p-\beta_1\\ \tilde\beta\end{pmatrix}\tilde\theta^{\tilde\beta^2}\tilde u^{\tilde\beta}\sum_{i,j=1}^{m_1-1} C_{i+1,j+1}\left(A_{i+1}\nabla \tilde u_i\right)\cdot\nabla \tilde u_j\nonumber\\
	&=\sum_{|\beta|=p-2}\begin{pmatrix}p\\\beta\end{pmatrix}\theta^{\beta^2}u^\beta\sum_{i,j=2}^{m_1}C_{i,j}\left(A_i\nabla u_i\right)\cdot\nabla u_j,
	\end{align}
	where the last step follows from the induction hypothesis. Now let's investigate I. We begin by expanding the $\nabla u^\beta$ term to find
	\begin{align}\label{Hp-eq11}
	\text{I}=&\sum_{\beta_1=0}^{p-1}\sum_{|\tilde\beta|=p-\beta_1-1}\begin{pmatrix}p\\\beta\end{pmatrix}\theta_1^{\beta_1^2}\tilde\theta^{\tilde\beta^2}\biggl[\theta_1^{2\beta_1+1}\left(A_1\nabla u_1\right)\cdot \bigg(\beta_1u_1^{\beta_1-1}\tilde u^{\tilde\beta}\nabla u_1+\sum_{i=1}^{m_1-1}\tilde\beta_iu_1^{\beta_1}\tilde u^{\tilde\beta-e_i}\nabla \tilde u_i\bigg)\nonumber\\
	&\qquad\qquad +\sum_{i=1,\beta_1\ne 0}^{m_1-1}\tilde\theta_i^{2\tilde\beta_i+1}\left(A_{i+1}\nabla \tilde u_i\right)\cdot\beta_1 u_1^{\beta_1-1}\tilde u^{\tilde\beta}\nabla u_1\biggr]=:\text{I}_{1,1}+\sum_{i=2}^{m_1}\text{I}_{i,1},
	\end{align}
	where 
	\begin{align}\label{Hp-eq12}
	\text{I}_{1,1}&=\sum_{\beta_1=1}^{p-1}\sum_{|\tilde\beta|=p-\beta_1-1}\begin{pmatrix}p\\\beta\end{pmatrix}\theta_1^{\beta_1^2}\tilde\theta^{\tilde\beta^2}\sum_{l=1}^n \theta_1^{2\beta_1+1}\beta_1u_1^{\beta_1-1}\tilde u^{\tilde\beta}\left(A_1\nabla u_1\right)_l\left(\frac{\partial u_1}{\partial x_l}\right)\nonumber\\
	&=\sum_{\beta_1=1}^{p-1}\sum_{|\tilde\beta|=p-(\beta_1-1)-2}\begin{pmatrix}p\\ (\beta_1-1,\tilde\beta) \end{pmatrix}\theta^{(\beta_1-1,\tilde\beta)^2} u^{(\beta_1-1,\tilde\beta)} \sum_{l=1}^n \theta_1^{4(\beta_1-1)+4}\left(A_1\nabla u_1\right)_l\left(\frac{\partial u_1}{\partial x_l}\right)\nonumber\\
	&=\sum_{\beta_1=0}^{p-2}\sum_{|\tilde\beta|=p-\beta_1-2}\begin{pmatrix}p\\ \beta \end{pmatrix}\theta^{\beta^2} u^{\beta} \sum_{l=1}^n \theta_1^{4\beta_1+4}\left(A_1\nabla u_1\right)_l\left(\frac{\partial u_1}{\partial x_l}\right)\nonumber\\
	&=\sum_{|\beta|=p-2}\begin{pmatrix}p\\ \beta \end{pmatrix}\theta^{\beta^2} u^{\beta} \sum_{l=1}^n \theta_1^{4\beta_1+4}\left(A_1\nabla u_1\right)_l\left(\frac{\partial u_1}{\partial x_l}\right),
	\end{align}
	and for $i\in\{2,...,m_1\}$,
	\begin{align}\label{Hp-eq13}
	\text{I}_{i,1}&=\sum_{\beta_1=0}^{p-1}\sum_{|\tilde\beta|=p-\beta_1-1}\begin{pmatrix}p\\\beta\end{pmatrix}\theta^{\beta^2}\sum_{l=1}^n \left[ \theta_1^{2\beta_1+1}\tilde\beta_{i-1}u_1^{\beta_1}\tilde u^{\tilde\beta-e_{i-1}}\left(A_1\nabla u_1\right)_l\frac{\partial u_i}{\partial x_l}\right.\nonumber\\
	&\quad \left. +\tilde\theta_{i-1}^{2\tilde\beta_{i-1}+1}\beta_1u_1^{\beta_1-1}\tilde u^{\tilde\beta}\frac{\partial u_1}{\partial x_l}\left(A_i\nabla u_i\right)_l\right]\nonumber\\
	&=\sum_{\beta_1=0}^{p-2}\sum_{|\tilde\beta|=p-\beta_1-1,\tilde\beta_{i-1}\ne 0}\begin{pmatrix}p\\(\beta_1,\tilde\beta-e_{i-1})\end{pmatrix}\theta^{\beta^2}\sum_{l=1}^n\theta_1^{2\beta_1+1}u_1^{\beta_1}\tilde u^{\tilde\beta-e_{i-1}}\left(A_1\nabla u_1\right)_l\frac{\partial u_i}{\partial x_l}\nonumber\\
	&\quad +\sum_{\beta_1=1}^{p-1}\sum_{|\tilde\beta|=p-\beta_1-1}\begin{pmatrix}p\\(\beta_1-1,\tilde\beta)\end{pmatrix}\theta^{\beta^2}\sum_{l=1}^n\theta_i^{2\beta_i+1}u_1^{\beta_1-1}\tilde u^{\tilde\beta}\frac{\partial u_1}{\partial x_l}\left(A_i\nabla u_i\right)_l\nonumber\\
	&=\sum_{|\beta|=p-2}\begin{pmatrix}p\\\beta\end{pmatrix}\theta^{\beta^2}u^\beta\sum_{l=1}^n\theta_1^{2\beta_1+1}\theta_i^{2\beta_i+1}\left(A_1\nabla u_1\right)_l\frac{\partial u_i}{\partial x_l}\nonumber\\
	&\quad +\sum_{|\beta|=p-2}\begin{pmatrix}p\\\beta\end{pmatrix}\theta^{\beta^2}u^\beta\sum_{l=1}^n\theta_1^{2\beta_1+1}\theta_i^{2\beta_i+1}\left(A_i\nabla u_i\right)_l\frac{\partial u_1}{\partial x_l}.
	\end{align}
	The result follows by combining (\ref{Hp-eq7}), (\ref{Hp-eq8}), (\ref{Hp-eq10}), (\ref{Hp-eq11}), (\ref{Hp-eq12}) and (\ref{Hp-eq13}). 
\end{proof}

\medskip
\noindent{\bf Acknowledgements:} We would like to thank the referees for their comments and suggestions, which improve the presentation of this paper.

The second author is partially supported by NAWI Graz.

\end{document}